\documentclass[a4paper,
11pt,reqno]{amsart}
\usepackage[utf8]{inputenc}

\usepackage{amsmath,amssymb}
\usepackage{amsthm}
\usepackage{mathtools}
\mathtoolsset{showonlyrefs}
\usepackage{xcolor}
\usepackage{pgfkeys,pgfcalendar}
\usepackage[british]{babel}

\usepackage{url}
\usepackage{enumitem}

\usepackage{CJKutf8}
\usepackage{csquotes}
\usepackage{mathrsfs}
\DeclarePairedDelimiter\floor{\lfloor}{\rfloor}

\RequirePackage[unicode]{hyperref}
\hypersetup{
	bookmarksopen,
	bookmarksopenlevel=1,
	pdfborder={0 0 0},
	pdfdisplaydoctitle,
	pdfpagelayout={SinglePage},
	pdfstartview={FitH},
	colorlinks=true,linkcolor=black,anchorcolor=black,citecolor=black,filecolor=black,menucolor=black,runcolor=black,urlcolor=blue
}
\usepackage[margin=3cm]{geometry}

\newtheorem{theorem}{Theorem}[section]
\newtheorem*{theorem*}{Theorem}
\newtheorem{lemma}[theorem]{Lemma}
\newtheorem{proposition}[theorem]{Proposition}

\theoremstyle{definition}
\newtheorem{definition}[theorem]{Definition}

\newtheorem{corollary}[theorem]{Corollary}
\newtheorem{Remark}[theorem]{Remark}

\theoremstyle{remark}
\newtheorem{remark}[theorem]{Remark}

\newcommand{\supp}{\mathop{\mathrm{supp}}}
\newcommand{\dd}{\mathrm{d}}
\newcommand{\ee}{\mathrm{e}}
\newcommand{\ve}{\varepsilon}
\newcommand{\vp}{\varphi}

\newcommand{\fin}{f_{\mathrm{in}}}
\newcommand{\gin}{g_{\mathrm{in}}}
\newcommand{\Min}{M_{\mathrm{in}}}
\numberwithin{equation}{section}

\newcommand{\Entr}{\mathcal{H}}
\newcommand{\dom}{{\mathbb{R}^d}}
\newcommand{\divv}{\mathrm{div}}
\newcommand{\stn}{^{(t_0)}}
\newcommand{\vt}{\vartheta}
\newcommand{\ul}{\underline}
\newcommand{\fpn}{\textsc{fp}$_\gamma$}
\newcommand{\fpnreg}{\textsc{fp}$_{\gamma.\mathrm{reg}}$}
\newcommand{\spi}{\delta}
\newcommand{\loc}{\mathrm{loc}}
\newcommand{\ptm}{a}

\newcommand{\weakstar}{\overset{\ast}{\rightharpoonup}}
\newcommand{\weak}{\rightharpoonup}

\newcommand{\RdnoO}{U}
\newcommand{\kqsign}{\sigma}
\newcommand{\tvmeas}{\mathcal{L}^{1+d}}
\newcommand{\kcp}{K_*}
\newcommand{\XX}{{X_{\ell,n}}}
\newcommand{\CC}{C_{\#}}

\makeatletter
\@namedef{subjclassname@2020}{\textup{2020} Mathematics Subject Classification}
\makeatother

\title[]
{Singularities in $L^1$-supercritical Fokker--Planck equations: A qualitative analysis}
\author[]
{Katharina Hopf}
\address{Katharina Hopf, Weierstrass Institute for Applied Analysis and Stochastics (WIAS),
	Mohrenstrasse 39, 10117 Berlin, Germany, 
	E-Mail: hopf@wias-berlin.de  
}


\AtBeginDocument{%
	\def\MR#1{}
}

\begin{document}

\begin{abstract}
	A class of nonlinear Fokker--Planck  equations with superlinear drift is investigated
	in the $L^1$-supercritical regime, which exhibits a finite critical mass.
	The equations have a formal Wasserstein-like gradient-flow structure 
	with a convex mobility and a free energy functional whose minimising measure has a singular component if above the critical mass. Singularities and concentrations also arise in the evolutionary problem and their finite-time appearance constitutes a primary technical difficulty.
	This paper aims at a global-in-time qualitative analysis
	with main focus on the isotropic case, where solutions will be shown to converge to the unique minimiser of the free energy as time tends to infinity.		
	A key step in the analysis consists in properly controlling the singularity profiles during the evolution.
Our study covers the 3D Kaniadakis--Quarati model for Bose--Einstein particles, and thus provides a first rigorous result on the continuation beyond blow-up and long-time asymptotic behaviour for this model.
\end{abstract}

	\maketitle
	
	\section{Introduction}
This manuscript is concerned with a class of Fokker--Planck equations with superlinear drift
taking the form
	\begin{equation}\label{eq:fpn}\tag{\fpn}
			\begin{aligned}
		\partial_tf&=\nabla\cdot\left(\nabla f+vh(f)\right),&& t>0, v\in \mathbb{R}^d,
		\\f(0,v)&=\fin(v)\ge0, &&v\in \mathbb{R}^d,
	\end{aligned}
	\end{equation}
where $h(f)=f(1+\kqsign |f|^\gamma)$ for some $\gamma\ge 1$ and $\kqsign=1$.
For $\gamma=1$ and $\kqsign\in\{\pm1\}$ this equation has been introduced in the 1990s by Kandiadakis and Quarati~\cite{KQ_1993,KQ_1994} as a model for the relaxation to equilibrium of quantum particles of Fermi--Dirac ($\kqsign=-1$) and Bose--Einstein ($\kqsign=1$) type.
We refer to~\cite{CHW_2020,Frank_2005} and references therein for more background on the physical model.
The interest of the mathematics community in problems of the form~\eqref{eq:fpn} mainly stems from their variational structure: 
for densities $f\ge 0$ the first line in~\eqref{eq:fpn} can formally be written as a continuity equation
\begin{align}\label{eq:gradflow}
	\partial_tf = \nabla\cdot\left(h(f)\nabla\delta \mathcal{H}(f)\right)
\end{align}
with $\delta \mathcal{H}$ denoting the variational derivative of the convex integral functional
\begin{align*}
	\mathcal{H}(f):=\int_\dom \left(\frac{|v|^2}{2}f+\Phi(f)\right)\,\dd v,
\end{align*}
where $\Phi(f):=\frac{1}{\gamma}\int_0^f\log\left(\frac{s^\gamma}{1+\kqsign s^\gamma}\right)\dd s$ and thus $\Phi''(f)=1/h(f)$. (If $\kqsign=-1$ one should restrict to $0\le f\le 1$.)
Thus,  the \textit{free energy} $\mathcal{H}(f)$ is formally dissipated along solutions
$\frac{\dd}{\dd t}\mathcal{H}(f)=-\int_\dom h(f)\left|\nabla\delta\mathcal{H}(f)\right|^2\,\dd v\le 0.$
Let us note that for $\kqsign=1$ the function $\Phi$ is sublinear at infinity, and the natural extension of $\mathcal{H}$ to finite, non-negative measures (cf.~\cite{DT_1984}) vanishes on Dirac deltas centred at the origin.
 We further observe that for $\kqsign=1$ the nonlinear mobility $h(f)=f(1+\kqsign f^\gamma)$ in~\eqref{eq:gradflow} is convex, while it is concave if $\kqsign=-1$.

The equation for fermions, where $\kqsign=-1$ and $\gamma=1$, is mathematically well-understood.
Here, in any dimension, solutions emanating from suitably regular initial data $0\le\fin\le1$
remain bounded  between zero and one, i.e.\ satisfy $0\le f\le 1$,  consistent with the well-known Pauli exclusion principle. In the long-time limit they converge to the unique minimiser of $\mathcal{H}$ of the given mass~\cite{CLR_2009,CRS_2008}, namely to the corresponding (smooth) Fermi--Dirac distribution.
The concavity of the mobility even allows to give a rigorous meaning to the above gradient-flow structure with respect to generalised Wasserstein distances~\cite{DNS_2009,CLSS_2010}, which fails for the convex/non-concave mobilities associated with $\kqsign=1$.

The bosonic case, where $\kqsign=1$ (and $\gamma=1$), is more challenging.
Then, equation~\eqref{eq:fpn} becomes $L^1$-supercritical in dimension $d>2$, in which case the large-data long-time analysis has remained open for quite a while.  In fact, a first global-in-time rigorous study of the $L^1$-supercritical regime has only recently been obtained in~\cite{CHR_2020} for a 1D analogue, that is for~\eqref{eq:fpn} with 
$\kqsign=1$, $d=1$ and $\gamma>2$, and is based on a Lagrangian approach and viscosity solution techniques.
 In the physically most interesting case $d=3$ and $\gamma=1$, which will be the main focus of this manuscript, no rigorous long-time analysis exists when $\kqsign=1$, except for the ref.~\cite{Toscani_2012} showing finite-time blow-up for large data by a virial-type contradiction argument.
In the $L^1$-critical case, in contrast, solutions are globally regular~\cite{CCLR_2016}. 
For numerical studies on the singularity formation in the supercritical case, we refer to \cite{CHW_2020,SSC_2006}.  The qualitative properties obtained in the present manuscript are in agreement with the simulations in~\cite{CHW_2020}, although our approximation scheme is different and not restricted to the isotropic case. Let us mention that the uniqueness and stability properties of the present scheme in the isotropic setting may also be of interest numerically. 

In this paper we perform a rigorous global-in-time existence and qualitative analysis of~\eqref{eq:fpn} with $\kqsign=1$ in the $L^1$-supercritical regime in higher dimensions $d\ge1$ our main interest being the bosonic 3D Kaniadakis--Quarati model ($\kqsign=1$ and $d=3,\gamma=1$); thus, hereafter $\kqsign=1$. 
Preservation of the variational structure beyond finite-time blow-up being a primary concern, 
we build our analysis on a suitably chosen approximation scheme that respects the basic mass conservation and structural properties of the continuity equation~\eqref{eq:gradflow}.
To begin with, we note that 
the stationary mass-constrained minimisation problem for $\mathcal{H}$
is well-understood.
 The minimisers of $\mathcal{H}$ for a given mass have been characterised in~\cite{BGT_2011} and are in fact explicit:

\begin{theorem*}[\cite{BGT_2011}, Theorem~3.1]\label{thm:entrMin}
	For every $m\in(0,\infty)$ the functional $\mathcal{H}$ 
	has a unique minimiser $\mu_{\min}=\mu^{(m)}_{\min}$ on the manifold
	$\{\mu\in\mathcal{M}_+(\mathbb{R}^d):\int \dd\mu=m\}.$\footnote{We define $\mathcal{H}(\mu):=\infty$ if $\int_{\mathbb{R}^d}|v|^2\dd\mu=\infty$.}	
	
	This minimiser takes the form
	\begin{align}\mu_{\min}=
		\begin{cases}
			f_{\infty,\theta}\,\mathcal{L}^d&\text{if }m\le m_c,\text{ where }\theta\in\mathbb{R}_{\ge0}\text{ obeys }\int_{\mathbb{R}^d} f_{\infty,\theta}(v)\,\dd v=m,\;\;
			\\f_c\,\mathcal{L}^d+(m-m_c)\,\delta_0&\text{if }m>m_c.
		\end{cases}\quad
		\label{eq:min}
	\end{align}
	Here
	\begin{align}\label{eq:den}
		f_{\infty,\theta}(v)
		=(\Phi')^{-1}\left(\;{-}\tfrac{1}{2}|v|^2{-}\theta\;\right)
		=\left(\mathrm{e}^{\gamma\big(\frac{|v|^2}{2}+\theta\big)}-1\right)^{-\frac{1}{\gamma}},\quad\theta\in\mathbb{R}_{\ge0},
	\end{align}
	and we abbreviated $f_c:=f_{\infty,0}$ as well as $m_c:=\int_{\mathbb{R}^d}f_c(v)\,\dd v\in(0,\infty]$.
\end{theorem*}

For general $\gamma\ge 1$, the $L^1$-supercritical regime as determined by a dimensional analysis is given by $d-\frac{2}{\gamma}>0$.
Observe that this is exactly the regime, where the \textit{critical mass} $m_c$ appearing in the above theorem is finite
 and where minimisers with singular parts concentrated at velocity zero emerge. 
 Such singular components are termed Bose--Einstein condensates in the physics literature (at least when $\gamma=1$).

Let us now put the analysis of the present work into context with existing literature and 
 discuss the main new difficulties. 
 Naturally, several aspects of our approach have their roots in the work~\cite{CHR_2020}. 
This is particularly true for the fact that our fundamental a priori bound 
consists in a space-uniform \textit{temporal} Lipschitz estimate (of an integral quantity)
that is propagated in time. Both, in~\cite{CHR_2020} and in the present paper, such estimates are derived by means of suitable comparison principles.
However, the approach in~\cite{CHR_2020} relies on a Lagrangian reformulation of the problem in terms of the pseudo-inverse cumulative distribution function giving access to the powerful instrument of viscosity solution theory~\cite{CIL_1992}. While in higher dimensions such a reformulation is, in principle, still possible~\cite{CHW_2020,CRW_2016,ESG_2005}, the structural properties of the resulting problem greatly deteriorate. Indeed, in higher space dimensions,  Lagrangian coordinates are vectorial and the corresponding reformulation leads to a quasilinear degenerate second-order \textit{system} of equations (cf.~\cite[Section~2.1.3]{CHW_2020}). In such situations, classical comparison techniques are rarely available. 
Even in the isotropic case, where a 1D scalar problem can be obtained for the inverse of a rescaled radial cumulative distribution function~\cite[Section~2.1.2]{CHW_2020},
the comparison technique in~\cite{CHR_2020} 
does not directly generalise to higher dimensions, since for $d>1$ the second-order differential operator in the new variables has an explicit dependence
on the unknown (see the comments following~(7.3) in~\cite{hopf_PhDThesis_2019}). 

The new challenges we encounter in higher dimensions are thus mainly of a technical nature. 
Especially the derivation of the universal space profile at $\{v=0\}$ for unbounded densities in Section~\ref{sec:profile} (applying to isotropic flows) is significantly more delicate than in the 1D case and requires several intermediate steps. 
Determining the profile at the \textit{first} blow-up time is still quite feasible and, as in the 1D case, amounts to solving an ordinary differential equation---in higher dimensions to be combined with a bootstrap argument.
However, in~\eqref{eq:fpn} solutions may regularise after a first blow-up, and such successions of \enquote{blow-ups} and \enquote{blow-downs} could in principle be highly oscillatory. Thus, for a global-in-time analysis a particular challenge lies in gaining information at  general points in time.
We should emphasize that the space profile, while of interest in its own right, encodes a certain time-uniform continuity-at-infinity property that appears to be vital for proving relaxation to the minimiser $\mu_{\min}$ in the long-time limit. (Observe that when only looking at the equation~\eqref{eq:fpn} from a PDE point of view, other stationary \enquote{solutions} consisting of a smooth steady state $f_{\infty,\theta}$ for some $\theta>0$ plus a suitably weighted non-trivial Dirac measure at zero are conceivable, though unphysical.)
Let us finally point out that, in contrast to~\cite{CHR_2020} where the mass of the condensate component (i.e.\ of the singular part of the measure solution, which turns out to be supported in $\{v=0\}$) has only been shown to be a continuous function of time, the present approach allows us to infer Lipschitz continuity in the isotropic case and thus refines~\cite{CHR_2020} (cf.~\cite{hopf_PhDThesis_2019}).
Some of the basic ideas of this manuscript have been sketched for the 1D model in the author's PhD Thesis~\cite[Chapter~5]{hopf_PhDThesis_2019}. As indicated in Chapter~5.3 of~\cite{hopf_PhDThesis_2019}, when $d=1$, the solutions to be constructed below coincide with those obtained from the viscosity solution approach in~\cite{CHR_2020}.

	\subsection{Main results}\label{ssec:main.results}
	
	In the subsequent analysis, unless specified otherwise, we assume the following general hypotheses:
	
	\medskip
	\begin{enumerate}[label=(H\arabic*)]
		\item\label{it:HPspc} $L^1$-supercriticality: $\frac{\gamma d}{2}>1$, where 
		$\gamma\in[1,\infty)$, $d\in\mathbb{N}_+$ are fixed parameters.
			\item\label{it:HP.init} Initial data: 
			\begin{itemize}[label=$\ast$]
				\item $\fin\ge0$ a.e.\ in $\mathbb{R}^d$.
		\item  Either $\fin\in (L^\infty_d\cap L^1_{2d+1})(\mathbb{R}^d)$ and $\fin$ is  isotropic,
				\\	or $\fin\in (L^\infty_\ell\cap L^1_{\ell+d+1})(\mathbb{R}^d)$ 
				 for some $\ell>3d+1$
				 is (possibly) anisotropic.
\end{itemize}
	\end{enumerate}
The spaces $L^p_\ell(\mathbb{R}^d)$ in~\ref{it:HP.init} are weighted  $L^p$ spaces with norm $\|f\|_{L^p_\ell}:=\|(1+|\cdot|^\ell)f\|_{L^p(\mathbb{R}^d)}$, see Section~\ref{ssec:notations}. 

 For the asymptotic analysis leading to Theorem~\ref{thm:profile.r} and all subsequent main results,  we further impose the hypothesis $\tfrac{2}{\gamma}+2-d>0$ (cf.~\eqref{eq:regime1}), which covers the most interesting case $\gamma=1$, $d=3$. For more details on this restriction, we refer to Remark~\ref{rem:regime1}.
		\medskip
	
	Our results for the nonlinear Fokker--Planck equations~\eqref{eq:fpn} rely on a careful analysis of the proposed approximation scheme, which is devised in such a way as to preserve the Fokker--Planck-type  gradient-flow structure~\eqref{eq:gradflow}. 
	Approximation schemes for continuation beyond blow-up have been employed in the literature for various other PDE problems. Closest to the present situation are perhaps the constructions in~\cite{LSV_2012,Velazquez_2004} for the 2D Patlak--Keller--Segel model.

\smallskip
\paragraph{\bf Approximation scheme.}\label{p:scheme}
Pick an even function $\eta\in C^{0,1}(\mathbb{R})\cap C^\infty(\mathbb{R}\setminus\{0\})$ that satisfies $\eta(s)=\eta(-s)$ for all $s\in\mathbb{R}$, 
 $\eta(s)=s^\gamma$ for $s\in[0,1]$, $\eta'(s)=0$ for $s\ge 2$, 
and which is further such that
$(0,\infty)\ni s\mapsto \frac{\eta(s)}{s^\gamma}$ is non-increasing.
For $\ve\in(0,1]$ we let $\eta_\ve(s):=\ve^{-\gamma}\eta(\ve s)$ and 
\begin{align}
	\begin{aligned}\label{eq:555}
		h_\ve(s)&:=s(1+\eta_\ve(s))
		\\&\;=:s+\vartheta_\ve(s),\quad \text{where }\vartheta_\ve(s):=s\eta_\ve(s).
	\end{aligned}
\end{align}
Note that the choice of $\eta$ implies that $h_\ve(s)\le h_{\ve'}(s)\le h(s)$ for all $s\ge0$ and $0<\ve'\le \ve\le 1$.
We then consider the associated Cauchy problem
\begin{align}\label{eq:fpnreg}\tag{\fpnreg}
	\begin{aligned}
		\partial_tf_\ve&=\nabla\cdot\left(\nabla f_\ve+vh_\ve(f_\ve)\right),&& t>0, v\in \mathbb{R}^d,
		\\f_\ve(0,v)&=\fin(v)\ge0, &&v\in \mathbb{R}^d.
	\end{aligned}
\end{align}
For details on the variational structure of~\eqref{eq:fpnreg} we refer to Section~\ref{sec:renorm}.
The global existence of non-negative mild solutions of~\eqref{eq:fpnreg} for suitably regular data can be  deduced using the linear growth of $h_\ve$ at infinity in conjunction with the fact that Fokker--Planck equations like~\eqref{eq:fpnreg} (and~\eqref{eq:fpn}) propagate moments of order higher than $2$ (cf.\ Prop.~\ref{prop:ex.reg} below). 
The relatively strong decay hypotheses in~\ref{it:HP.init} are primarily needed to establish estimates that are independent of $\ve$ (cf.\ Prop.~\ref{prop:cpbound}).
The notations $\mathcal{M}_+(\mathbb{R}^d), \mathcal{L}^d$ and further conventions used in the following proposition are specified in Section~\ref{ssec:notations}.

\begin{proposition}[Limiting measure for~\eqref{eq:fpn}]\label{prop:lim}
Suppose~\ref{it:HPspc}, \ref{it:HP.init}.
	Then there exists a non-negative Radon measure $\mu$ on $[0,\infty)\times\mathbb{R}^d$
	with the following properties:
	 \begin{enumerate}[label=\textup{(\roman*)}]
	 	\item\label{it:curve} Mass-conserving curve: 
	 	$\mu$ can be represented as $\dd\mu=\dd\mu_t\dd t$ for a family of measures $\{\mu_t\}_{t\ge0}\subset\mathcal{M}_+(\mathbb{R}^d)$ with the property that 
	 	$t\mapsto\mu_t$ is a weakly-$\ast$ continuous curve in $\mathcal{M}_+(\mathbb{R}^d)$
	 	that satisfies $\mu_t(\mathbb{R}^d)=\|\fin\|_{L^1}=:m$ for all $t\ge0$ and admits a decomposition according to~\ref{it:dec}.
	 	\item\label{it:dec} Decomposition: there exists a measurable function $\ptm:[0,\infty)\to[0,m]$ and
	 	a non-negative function $f\in  L^1_\loc([0,\infty)\times\mathbb{R}^d)\cap C^{1,2}((0,\infty)\times U),$ $U:=\mathbb{R}^d\setminus\{0\},$ such that for all $t\ge0$
	 	\begin{align}
	 		\mu_t = \ptm(t)\delta_0+f(t,\cdot)\mathcal{L}^d,
	 	\end{align}
 	where $\delta_0$ denotes the Dirac measure concentrated at the origin.\\[1mm]
The function $f$ is a classical solution of~\eqref{eq:fpn} in $(0,\infty)\times U$.
Moreover, $f$ is strictly positive in $(0,\infty)\times\mathbb{R}^d$ if $\|\fin\|_{L^1}>0$
in the sense that for all $K\subset\subset (0,\infty)\times\mathbb{R}^d$ there exists $c(K)>0$ such that 
$f_{|K}\ge c(K)$ a.e.\ in $K$. 
 	 	\item\label{it:approxProp}  Approximation property: denote by $f_\ve\in C([0,\infty);(L^\infty_1\cap L^1_3)(\mathbb{R}^d))$ the unique mild solution\footnote{The approximate solutions $f_\ve$  enjoy further regularity properties, which will be needed in the analysis; see Section~\ref{sec:ex} for details.} of~\eqref{eq:fpnreg} (cf.\ Sec.~\ref{ssec:mild}) and let $\mu^{(\ve)}=f_\ve\tvmeas_+$, where $\tvmeas_+$ denotes the $(1{+}d)$-dimensional Lebesgue  measure on $[0,\infty)\times\mathbb{R}^d$. \\
 	 	Then, along a subsequence $\ve\downarrow 0$
 	 	\begin{align}
 	 		&\mu^{(\ve)}\overset{\ast}{\rightharpoonup} \mu\text{ in }\mathcal{M}_+([0,T]\times\mathbb{R}^d)\quad\text{ for all }T<\infty,
 	 		\\&	f_\ve\to f \text{ in  }\;C^{1,2}_{\loc}((0,\infty)\times U),
 	 	\end{align}
  	where $U:=\mathbb{R}^d\setminus\{0\}$. 
  	\item\label{it:uniq.sola} Unique limit: if $\fin$ is isotropic, the convergence in~\ref{it:approxProp} is 
  	true along any sequence $\ve\downarrow0$. 
  	\item\label{it:lip.ptm} Lipschitz continuity of point mass: if $\fin$ is isotropic\footnote{In the anisotropic case, we will see in Section~\ref{sec:profile} that $t\mapsto\mu_t(\{0\})$ is at least continuous, see Cor.~\ref{cor:upper.ani}.}, the map $t\mapsto\mu_t(\{0\})$ is Lipschitz continuous.
	 \end{enumerate}
\end{proposition}

See Section~\ref{ssec:passlim} for the proof of Proposition~\ref{prop:lim}.
Later on we show for the isotropic case that the limiting measure $\mu$ satisfies~\eqref{eq:fpn} is the sense of renormalised solutions.
One of the technical difficulties of problem~\eqref{eq:fpn} is related to the fact that the function $t\mapsto\mu_t(\{0\})$ in general fails to be monotonic (cf. Sec.~\ref{ssec:dynprop}).

Proposition~\ref{prop:lim}~\ref{it:dec} implies that $\supp\mu_t^{\mathrm{sing}}\subset\{v=0\}$. 
Hence, recalling the sublinearity of $\Phi(s)$ as $s\to\infty$, we infer that for every $t\ge0$
\begin{align}
	\mathcal{H}(\mu_t)= \mathcal{H}(f(t)).
\end{align}
Since all relevant measures in this work will have singular parts supported at the origin, we (continue to) denote by the symbol $\mathcal{H}$ both the functional acting on densities as well as the extended functional acting on non-negative finite measures.

The following result provides a sharp characterisation of the space profile near the origin of isotropic solutions, and moreover it is a key ingredient for uniquely identifying the long-time asymptotic limit. It will be established in Section~\ref{sec:profile}.

\begin{theorem}[Universal space profile]\label{thm:profile.r}
	In addition to~\ref{it:HPspc}, \ref{it:HP.init}	suppose that
			\begin{align}\label{eq:regime1}
	\tfrac{2}{\gamma}+2-d>0.
		\end{align} 
	Further assume that the initial value $f_\mathrm{in}$ is isotropic and let $g(t,|v|):=f(t,v)$, where $f$ denotes the density of the regular part of the limiting measure obtained in Proposition~\ref{prop:lim}.
There exists $r_*\in(0,1]$ and a bounded function 
$A\in C_b((0,\infty)\times (0,1))$
 such that 
	for each $\hat t>0$ 
	\emph{either}
	$g(\hat t,\cdot)\in L^\infty(0,1)$ 
	 and there exists a neighbourhood $J_{\hat t}\subset(0,\infty)$ of $\hat t$ 
	such that $f_{|J_{\hat t}\times B_1}$ is smooth
		\emph{or}
	\begin{align}\label{eq:profile*}
	g(\hat t,r)= g_c(r) +A(\hat t,r)r^{2-d} \qquad\text{ for }\; r\in(0,r_*),\hspace{-2cm}
	\end{align}
where $g_c(|v|):=f_c(v)=f_{\infty,0}(v)$ (cf.~\eqref{eq:den}), i.e.\ 
$g_c(r)=(\Phi')^{-1}(-\frac{1}{2}r^2)$. \\
The radius $r_*\,{>}\,0$ and the function $A\,{\in} \,C_b((0,\infty){\times} (0,1))$ can be chosen in such a way that 
\begin{align}\label{eq:thm.profile.r.upper.bd}
	g(\hat t,r)\le g_c(r) +A(\hat t,r)r^{2-d}\text{ for all }\; r\in(0,r_*)\text{ and all }\hat t\in(0,\infty).
\end{align}
For all $\hat t>0$ satisfying $\mu_{\hat t}(\{0\})>0$ the second option, i.e.~\eqref{eq:profile*}, must hold true.
\end{theorem}
See Section~\ref{ssec:profile.proof} for the proof of Theorem~\ref{thm:profile.r}.
The main challenge is to show the lower bound $g(\hat t,r)\ge g_c(r) +A(\hat t,r)r^{2-d}$, $r\in(0,r_*)$, encoded in~\eqref{eq:profile*}.
For its proof we combine different techniques:
based on the global temporal Lipschitz continuity of the partial mass function, 
we first establish a partial result on the \enquote{stability from below} of the unbounded steady state~$f_c$ by employing a bootstrap argument that bears some elements of classical intersection comparison~\cite{Galaktionov_2004,GV_2004}.
This step strongly relies on the radial symmetry assumption. It allows to infer~\eqref{eq:profile*} at times~$\hat t$ where $\mu_{\hat t}(\{0\})>0$, for instance. The full characterisation in Theorem~\ref{thm:profile.r} is only achieved upon a combination with specially tailored semi-group estimates for mild solutions along with a contradiction-type argument.
 We refer to Remark~\ref{rem:lb} for more details. 
 The upper bound~\eqref{eq:thm.profile.r.upper.bd} 
 also holds in the anisotropic case (see Remark~\ref{rem:ani} and Corollary~\ref{cor:upper.ani}). 

We further note that $g_c(r)=\big(\tfrac{2}{\gamma}\Big)^\frac{1}{\gamma}r^{-\frac{2}{\gamma}}+O(r^{-\frac{2}{\gamma}+2})$ for $0<r\ll1$, so that the remainder $O(r^{2-d})$ in~\eqref{eq:profile*} is indeed of lower order under condition~\eqref{eq:regime1}. Moreover, in the expansion for $g$ one can replace the limiting steady state $g_c(r)$ by the power law $\big(\tfrac{2}{\gamma}\Big)^\frac{1}{\gamma}r^{-\frac{2}{\gamma}}$ since $d>\frac{2}{\gamma}$.
\begin{remark}[The regime~\eqref{eq:regime1}]\label{rem:regime1}
In the present work, we focus on the range~$\tfrac{2}{\gamma}+2-d>0$ as it covers the most interesting case of the 3D Kaniadakis--Quarati model for bosons ($\gamma=1$, $d=3$).  If  $\tfrac{2}{\gamma}+2-d<0$, 
the flux into the origin associated with the nonlinear drift term $\divv(vh(f_c))$ induced by the unbounded steady state $f_c$ vanishes in the sense that $\lim_{r\downarrow0}\int_{\partial B_r(0)}h(f_c)v{\cdot}\nu\,\dd\mathcal{H}^{d-1}=0$. Here, $\mathcal{H}^{d-1}$ denotes the $(d{-}1)$-dimensional Hausdorff measure and $\nu$ the outer unit normal to $\partial B_r(0)$, so that $v{\cdot}\nu=r$ for $v\in\partial B_r(0)$.
Heuristically, this suggests that an upper bound of order $f_c$ on the space profile of the density $f(t,\cdot)$ near zero (as asserted in Theorem~\ref{thm:profile.r} for regime~\eqref{eq:regime1}) might not be compatible with the formation of a point mass at the origin in the case $\tfrac{2}{\gamma}+2-d<0$.
 In view of Theorem~\ref{thm:long-time}, asserting in particular the formation of a Dirac mass  in finite time for mass-supercritical data, we conjecture that some new phenomena 
may be encountered when~\eqref{eq:regime1} is violated.
\end{remark}

Owing to the strong nonlinearity in the drift one cannot expect the limiting density~$f$ to be a distributional solution of~\eqref{eq:fpn} in $(0,\infty)\times\mathbb{R}^d$. 
Our analysis leading to Theorem~\ref{thm:profile.r} allows to show that the limiting measure satisfies~\eqref{eq:fpn} in the sense of renormalised solutions.
\begin{definition}[Renormalised solution of~\eqref{eq:fpn}]\label{def:renorm}
	Let $\mu$ be a non-negative Radon measure on $[0,\infty)\times\mathbb{R}^d$
	and denote by $\mu=\mu^\mathrm{reg}+\mu^\mathrm{sing}=f(t,v)\mathcal{L}^{1+d}_++\mu^\mathrm{sing}$
	its Lebesgue decomposition into regular part $\mu^\mathrm{reg}$ with density $f\in L^1_\loc([0,\infty)\times\mathbb{R}^d)$ and singular part $\mu^\mathrm{sing}$. 
	We call $\mu$ a renormalised solution of~\eqref{eq:fpn} in $(0,\infty)\times\mathbb{R}^d$ with initial data $\fin$ if $\dd\mu=\dd\mu_t\dd t$ for some weakly-$\ast$ continuous curve $[0,\infty)\ni t\mapsto\mu_t$ in $\mathcal{M}_+(\mathbb{R}^d)$ with preserved mass $\int\dd\mu_t\equiv\|\fin\|_{L^1(\mathbb{R}^d)}$, if
	$\mathcal{T}_k(f):=\min\{f,k\}\in L^2_\loc([0,\infty);H^1_\loc(\mathbb{R}^d))$ for every $k>0$,
	and if for all $\xi\in C^\infty([0,\infty))$ with compactly supported derivative $\xi'$, for a.a.\ $T\in(0,\infty)$ and all $\psi\in C^\infty_c([0,T]\times\dom):$
\begin{equation}\label{eq:renorm}
		\begin{multlined}
			\int_\dom \xi(f(T,\cdot))\psi(T,\cdot)\,\dd v-\int_\dom \xi(\fin)\psi(0,\cdot)\,\dd v
		-\int_0^T\!\!\int_\dom \xi(f)\partial_t\psi\,\dd v\dd t
	\\	=-\int_0^T\!\!\int_\dom (\nabla f+h(f)v)\cdot \nabla(\xi'(f)\psi)\,\dd v\dd t.
	\end{multlined}
\end{equation}
\end{definition}
As usual, the gradients of $f$ on the RHS of~\eqref{eq:renorm} are to be understood as 
$\nabla \mathcal{T}_k(f)$ for $k=k(\xi)$ large enough such that $\xi'(s)=0$ for $s\ge k$ (cf.~\cite{BBGGV_1995,DalMMOP_1999}).

Let us emphasise that the  above definition of renormalised solutions should be seen as preliminary.
For a `better' and more complete paradigm, the solution concept may have to be complemented by suitable \textit{energy} or \textit{entropy conditions} as it is classical for conservation laws and nonlinear elliptic/parabolic equations, see~\cite{Kr_1970,CDiFT_2016,BBGGV_1995,BM_1997}, where they are crucial for uniqueness. 
A general analysis of the question of uniqueness for~\eqref{eq:fpn} is, however, beyond the scope of the present manuscript and will be left for future research.

\begin{theorem}[The limit $\mu$ is a renormalised solution]\label{thm:renorm}
Assume the hypotheses of Theorem~\ref{thm:profile.r}. 
Then the limiting measure $\mu$ constructed 
in Proposition~\ref{prop:lim} satisfies~\eqref{eq:fpn} in the renormalised sense as specified in Definition~\ref{def:renorm}.
\end{theorem}
The proof of this theorem is given in Section~\ref{ssec:renorm} and 
makes use, among others, of a local and truncated version of the energy dissipation estimate.
The following energy dissipation identity is crucial for deducing the long-time asymptotic behaviour.

\begin{proposition}[Energy dissipation (in)equality]\label{prop:edb}
Assume~\ref{it:HPspc}, \ref{it:HP.init} and use the notations of Proposition~\ref{prop:lim}.
Then for all $t>0$
		\begin{align}\label{eq:ed}
		\Entr(f(t)) - \Entr(\fin) \le  -\int_0^t\!\int_\dom \frac{1}{h(f)}|\nabla f+h(f)v|^2\,\dd v\dd\tau.
	\end{align}
When supposing in addition the hypotheses of Theorem~\ref{thm:profile.r}, the stronger   balance law holds true:
for all $t\ge s\ge 0$
\begin{align}\label{eq:edb}
	\Entr(f(t)) - \Entr(f(s)) = -\int_s^t\!\int_\dom \frac{1}{h(f)}|\nabla f+h(f)v|^2\,\dd v\dd\tau.
\end{align}
\end{proposition}
See Section~\ref{ssec:edi} for the proof of Proposition~\ref{prop:edb}.

The long-time behaviour and further transient dynamical properties can be seen as corollaries of the above results (cf.\ Section~\ref{sec:longtime} for details). 
Let us here only highlight the long-time asymptotics. 
\begin{theorem}[Convergence to minimiser]\label{thm:long-time}
	Assume the hypotheses of  Theorem~\ref{thm:profile.r} and denote by $m=\int\fin>0$ the total mass of the initial data. Further let $\mu_{\min}:=\mu^{(m)}_{\min}$ denote the unique minimising measure of $\mathcal{H}$ for the given mass $m$ (cf.\ eq.~\eqref{eq:min}).
	Then,  as $t\to\infty$, $\mathcal{H}(\mu_t)\to\mathcal{H}(\mu_{\min})$, and moreover
		\begin{align}
		&\mu_t\weakstar\mu_{\min}\quad\text{ in }\mathcal{M}_+(\mathbb{R}^d)\qquad \text{and}
		\qquad \mu_t(\{0\})\to\mu_{\min}(\{0\}),
		\\&f(t)\to f_{\min}\quad\text{ in }C^2_\loc(\mathbb{R}^d\setminus\{0\}),
		\\&f(t)\to f_{\min}\quad\text{ in }L^p(\mathbb{R}^d)\quad\text{ for any }p\in[1,\tfrac{\gamma d}{2}),
	\end{align}
where $f_{\min}$ denotes the density of the regular part of $\mu_{\min}$ with respect to the Lebesgue measure.
\end{theorem}
The proof of this result will be completed in Section~\ref{ssec:relax}.

\begin{Remark}[Anisotropic case]\label{rem:ani}
	While the main conclusions in Theorems~\ref{thm:profile.r},~\ref{thm:renorm} and~\ref{thm:long-time} are restricted to the isotropic setting, several of the 	
	intermediate results derived in this paper are proved for anisotropic data. 
	Below, we summarise the most relevant results obtained for anisotropic data satisfying~\ref{it:HPspc},~\ref{it:HP.init}, where for simplicity we restrict to regime~\eqref{eq:regime1}.
	\begin{enumerate}[label=(\alph*),leftmargin=1.75\parindent]
		\item\label{it:ani.mc-curve} Limiting measure \& upper bound for density: Convergence of a sequence of approximate solutions to a mass-conserving curve $t\mapsto \mu_t= \ptm(t)\delta_0+f(t,\cdot)\mathcal{L}^d$ as detailed in Proposition~\ref{prop:lim}~\ref{it:curve}--\ref{it:approxProp}. The density $f(t,\cdot)$ satisfies the pointwise bound 
		\[\;\;\qquad0\le f(t,v) \le 
		\big(\tfrac{2}{\gamma}\Big)^\frac{1}{\gamma}|v|^{-\frac{2}{\gamma}}
		\min\{\, 1\,{+}\,C_1|v|^{\frac{2}{\gamma}+2-d},C_2|v|^{-(d-\frac{2}{\gamma})}\,\}\;\;\text{for all $t>0$, $v\in\mathbb{R}^d$},\] and the point mass at the origin 
		$t\mapsto\mu_t(\{0\})$ is continuous, see Corollaries~\ref{cor:upper.ani}~and~\ref{cor:boundaniso}.
		\item\label{it:ani.edi} Energy dissipation inequality: The density $f$ obeys inequality~\eqref{eq:ed} for all $t>0$.
		\item\label{it:ani.asympt} Finite-time condensation and relaxation to free energy minimiser for certain data:
		If $\fin$ is bounded below by an admissible isotropic density of supercritical mass, the limiting measure obtained in~\ref{it:ani.mc-curve} exhibits a Dirac mass  at the origin after a finite time and converges, in the long-time limit, to the (singular) minimiser of the free energy with the same mass.
	\end{enumerate}
The crucial compactness property leading to the convergence result in~\ref{it:ani.mc-curve} is obtained by pointwise comparison at the level of the approximate densities with an isotropic envelope. It is at this point, where the stronger decay hypothesis on the data imposed in the anisotropic case (cf.~\ref{it:HP.init}) enters, since for this argument the initial data $\fin$ need to lie below an admissible isotropic envelope $\hat f_{\rm in}$, that is $\fin\le\hat f_{\rm in}$. The assertion in~\ref{it:ani.asympt} follows from similar comparison arguments combined with the time-asymptotic results obtained in the isotropic case.
Finally, key to the energy dissipation inequality~\ref{it:ani.edi} are the Fokker--Planck type variational structure of the regularised problem, lower semi-continuity properties of the convex free energy and 
the strong convergence properties away from the origin in Proposition~\ref{prop:lim}~\ref{it:approxProp}.
This paper leaves open the question of whether the limiting measure obtained in~\ref{it:ani.mc-curve} relaxes to the minimiser of the free energy with the same mass for all admissible anisotropic data. While the energy dissipation inequality allows to conclude convergence of the density $f$ along a sequence of times $t_k\to\infty$ to $f_{\infty,\theta}$ for some $\theta\ge0$, our estimates do not allow to rule out the case of the parameter $\theta$ being larger than that of $\mu_{\min}$ for a given mass, i.e.\ the case of a simultaneous presence of a smooth density and a Dirac mass at zero. 
The main problem is the lack of sufficiently strong lower bounds for anisotropic densities in the presence of singularities and concentrations. Such bounds might possibly be obtained by suitably controlling the change of mass in small neighbourhoods of the origin. In the isotropic case, the crucial estimate is~\eqref{eq:apriori}. It follows from the comparison principle structure of the equation for the partial mass function~\eqref{eq:208}.
	\end{Remark}
	\subsection{Outline}
	The remaining part of this paper is structured as follows. In Section~\ref{sec:ex} we establish global existence  for the approximate problem~\eqref{eq:fpnreg} as well as uniform estimates, which 
	allow us to pass to the limit $\ve\to0$ in Section~\ref{ssec:passlim}. An
	 important ingredient is the uniform bound in Proposition~\ref{prop:cpbound}, which is obtained using a comparison technique. 
	Section~\ref{sec:profile} lies at the heart of our analysis. Its main purpose is to establish the universal profile asserted in Theorem~\ref{thm:profile.r} (see Sec.~\ref{ssec:profile.proof}).
In Section~\ref{sec:renorm} we introduce entropy tools and use the results from Section~\ref{sec:profile} to show, for the isotropic case, the renormalised solution property of~\eqref{eq:fpn} as well as the energy dissipation identity.
Section~\ref{sec:longtime} concludes with a characterisation of the long-time asymptotics and some additional remarks.
	
\subsection{Notations}\label{ssec:notations}	
Unless specified otherwise, we adopt the following notations:
	\begin{itemize}[label=\raisebox{0.25ex}{\tiny$\bullet$},leftmargin=1.1\parindent]\itemsep.1em
		\item 
		$L^p_\ell(\mathbb{R}^d)$ for $p\in[1,\infty],\ell\ge0:$ weighted $L^p$-space with norm $\|f\|_{L^p_\ell}:=\|(1+|\cdot|^\ell)f\|_{L^p(\mathbb{R}^d)}$,  
		where $|\cdot|$ denotes the function $v\mapsto|v|$. The 
		spaces $L^p_\ell(\mathbb{R}^d)$ are Banach spaces.
		\item $C^{1,2}((0,\infty)\times\mathbb{R}^d):$ space of continuously differentiable  functions $f=f(t,v)$ that are twice continuously differentiable with respect to $v\in \mathbb{R}^d$.
		\item $\mathcal{M}_+(G):$ set of non-negative finite (Radon) measures on a given Borel set $G\subset\mathbb{R}^N$, $N\in \mathbb{N}$. Usually, $G=\mathbb{R}^d$ or $G=I\times\mathbb{R}^d$ for an interval $I\subset[0,\infty)$.
		\item $\mu_n\weakstar\mu$ in $\mathcal{M}_+(G)$ for $\mu_n,\mu\in \mathcal{M}_+(G)$
		 stands for the convergence $\int_G \vp\,\dd\mu_n\to\int_G \vp\,\dd\mu$ for all $\vp\in C_b(G)$. 
		 This mode of convergence will be referred to as weak-$\ast$ convergence of measures.
		 It is induced by a distance on $\mathcal{M}_+(G)$~\cite[Sec.~5.1]{AGS_2008},~\cite{Klenke_2014}.
		\item We write $\dd\mu=\dd\mu_t\dd t$ for non-negative Radon measures $\mu$ on $[0,\infty)\times\mathbb{R}^d$ and $\mu_t$ on $\mathbb{R}^d$, $t\ge0$, 
		with $\mu_t(\mathbb{R}^d)\equiv \mathrm{const.}\in\mathbb{R}_+$  	
		 if for every $\vp\in C_c([0,\infty)\times\mathbb{R}^d)$ the function $t\mapsto\int_{\mathbb{R}^d}\vp(t,v)\,\dd\mu_t(v)$ is 
		Lebesgue measurable and 
		$\int_{[0,\infty)\times\mathbb{R}^d}\vp\,\dd\mu=\int_{[0,\infty)}\int_{\mathbb{R}^d}\vp(t,v)\,\dd\mu_t(v)\dd t$.
		 \item $\mu^\mathrm{reg}$, $\mu^\mathrm{sing}:$ regular and singular part of $\mu\in \mathcal{M}_+(G)$ with respect to the Lebesgue measure on $G\subset\mathbb{R}^N$.
		 \item $\mathcal{L}^d:$  $d$-dimensional Lebesgue measure.
		 \item $\tvmeas_+:$  
		 $(1{+}d)$-dimensional Lebesgue measure on $[0,\infty)\times\mathbb{R}^d$.
		\item $A\lesssim B$ for non-negative quantities $A,B$ stands for $A\le CB$ for a fixed constant $C\in(0,\infty)$. The relation $A\gtrsim B$ is defined as $B\lesssim A$.
		\item $s_+:=\max\{s,0\}$ for $s\in\mathbb{R}$.
		\item $B_r:=B_r(0):=\{v\in \mathbb{R}^d:|v|<r\}$.
		\item $g_c(r)=f_c(v)$ for $r=|v|$, where $f_c=f_{\infty,0}$ as defined in~\eqref{eq:den}.
		Equivalently, $g_c(r)=(\Phi')^{-1}(-\frac{1}{2}r^2)$.
	\end{itemize}
	
	\section{Approximation scheme}\label{sec:ex}
	
As pointed out in the introduction,  local-in-time classical solutions of~\eqref{eq:fpn} emanating from initial data that are large in a suitable sense may cease to exist in $L^\infty(\mathbb{R}^d)$ after a finite time. 
The main purpose of this section is to establish global existence for the approximation scheme~\eqref{eq:fpnreg} in spaces of suitable regularity as well as certain compactness and convergence properties for the corresponding approximate solutions.
In the isotropic case, our scheme obeys a monotonicity property and, as a consequence, gives rise to a unique limiting measure. Note that this feature may also be of interest from a numerics point of view. A key ingredient in the analysis is a uniform temporal Lipschitz bound for the partial mass function of isotropic solutions (see Proposition~\ref{prop:cpbound}).

	\subsection{Mild solutions}\label{ssec:mild}

The local-in-time wellposedness of equations~\eqref{eq:fpn} and~\eqref{eq:fpnreg} in suitably weighted spaces can conveniently be obtained in the framework of mild solutions
using the Duhamel integral formulation  of~\eqref{eq:fpn} resp.\ of~\eqref{eq:fpnreg} given by
{\small
\begin{align}\label{eq:115ori}
		f(t,v) &=\int_{\mathbb{R}^d}\mathcal{F}(t,v,w)\fin(w)\,\dd w
		+\int_0^t\!\int_{\mathbb{R}^d}\mathcal{F}(t{-}s,v,w)\big(\divv_w(w\,|f|^{\gamma}f)\big)|_{(s,w)}\,\dd w\dd s,\qquad
\\\label{eq:115}
	f_\ve(t,v) &=\int_{\mathbb{R}^d}\mathcal{F}(t,v,w)\fin(w)\,\dd w
	+\int_0^t\!\int_{\mathbb{R}^d}\mathcal{F}(t{-}s,v,w)\big(\divv_w(w\,\vartheta_\ve(f_\ve))\big)|_{(s,w)}\,\dd w\dd s,\quad
\end{align}}where $\mathcal{F}=\mathcal{F}(t,v,w)$ denotes the fundamental solution of the linear Fokker--Planck equation
$\partial_t\mathsf{f}=\nabla\cdot(\nabla \mathsf{f}+v\mathsf{f})$, that is~(cf.~\cite{CT_1998})
\begin{align}
	\mathcal{F}(t,v,w) = \ee^{dt}G_{\nu(t)}(\ee^tv-w)
\end{align}
with
\begin{align}
	\nu(t)=\mathrm{e}^{2t}-1,\qquad 
	G_\lambda(\xi)=(2\pi\lambda)^{-\frac{d}{2}}\mathrm{e}^{-\frac{|\xi|^2}{2\lambda}}.
\end{align}
 In this subsection, we collect several auxiliary results for mild solutions, many of which can be obtained as in~\cite{CLR_2009}. 
The reasoning is therefore kept brief.

Using integration by parts, equation~\eqref{eq:115ori} can formally be rewritten as
\begin{align}\label{eq:116}
	\begin{aligned}
f(t,v) &=\int_{\mathbb{R}^d}\mathcal{F}(t,v,w)\fin(w)\,\dd w\\&\qquad +\int_0^t\ee^{-(t-s)}\int_{\mathbb{R}^d}
	\nabla_v\mathcal{F}(t{-}s,v,w)\cdot w\,|f|^{\gamma}f|_{(s,w)}\,\dd w\,\dd s.
\end{aligned}
\end{align}
Analogously, we may rewrite equation~\eqref{eq:115}. For estimating the integrals appearing on the right-hand side of~\eqref{eq:116} we use the semi-group estimates in~\cite[Appendix~A]{CLR_2009}.
By~\cite[Proposition~A.1]{CLR_2009} the linear operator
\begin{align}
	\mathscr{F}[f](t,v):=\int_{\mathbb{R}^d}\mathcal{F}(t,v,w)f(w)\,\dd w
\end{align}
enjoys the following smoothing estimates for all $t\in(0,T]$ and $T<\infty$
\begin{align}\label{eq:FPsemigr}
	\|\nabla^k_v\mathscr{F}[f](t)\|_{L_\ell^q}\le C_T\nu(t)^{-\frac{d}{2}(\frac{1}{p}-\frac{1}{q})-\frac{k}{2}}\|f\|_{L^p_\ell}
\end{align}
for any $1\le p\le q\le \infty$, $\ell\ge0$ and $k\in \mathbb{N}_0$, where
the constant $C_T=C_T(d,q,k)$ is given by $C_T=C\exp((\tfrac{d}{q'}+k)T)$
 with $\frac{1}{q'}+\frac{1}{q}=1$ and $C<\infty$ a universal constant.
 For the definition of $\|\cdot\|_{L^p_\ell}$ we refer to Section~\ref{ssec:notations}.
 
In the rest of this section, $C$ denotes a constant that may depend on fixed parameters, but not on time. Constants that additionally depend on the (final) time $T<\infty$ are denoted by $C_T$.
 Any such constants may change from line to line.

We begin with a uniqueness result.
\begin{lemma}[Uniqueness of mild solutions for~\eqref{eq:fpn} and~\eqref{eq:fpnreg}]\label{l:uniq.mild}
	Let $p>d$.  
	There exists at most one mild solution $f\in C([0,T];(L^\infty\cap L^p_1)(\mathbb{R}^d))$
	of equation~\eqref{eq:fpn}. 
	 An analogous result holds for equation~\eqref{eq:fpnreg}.
\end{lemma}

\begin{proof}
	Let $f,\tilde f\in C([0,T];(L^\infty\cap L^p_1)(\mathbb{R}^d))$ both 
	satisfy equation~\eqref{eq:116} for $t\in[0,T]$. 
	Note that since $p>d$, we have $\alpha:=\frac{1}{2}+\frac{d}{2}\frac{1}{p}\in[0,1)$.
	We may thus estimate for $t\in[0,T]$, using the bound~\eqref{eq:FPsemigr} and recalling that $\nu(t)=\ee^{2t}-1$,
	\begin{align}
		\|f(t)-\tilde f(t)\|_{L^\infty}&\le C_T\int_0^t\nu(t{-}s)^{-\alpha}
		\|w(|f|^{\gamma}f(s,w)-|\tilde f|^{\gamma}\tilde f(s,w))\|_{L^p}\,\dd s
	\\&\le  C_T\||f|+|\tilde f|\|_{C([0,T];L^p_1\cap L^\infty)}^\gamma
	\int_0^t(t{-}s)^{-\alpha}
	\|f(s)-\tilde f(s)\|_{L^\infty}\,\dd s,
	\end{align}
where we used the fact that $\nu(t)\ge 2t$.
	Invoking the singular Gronwall inequality (see e.g.~\cite[Theorem~3.3.1]{Amann_1995}),
	we infer that $f(t)= \tilde f(t)$ for all $t\in(0,T]$, which shows the asserted uniqueness.
\end{proof}

We now seek to construct solutions taking values in the Banach space
\begin{align}\label{eq:defX.r}
	\XX =(L^\infty_\ell\cap L^1_n)(\mathbb{R}^d)
\end{align}
for  sufficiently large $\ell,n\in[1,\infty)$. 
Global-in-time existence of non-negative mild solutions to~\eqref{eq:fpnreg} will be obtained 
under the additional hypothesis that $n=n(\ell,d)$ be sufficiently large, cf.~\eqref{eq:paramX.r}. 
The reason for this condition is that
 our approach to control the $L^\infty_\ell$-norm of a local solution relies on an a priori control in $L^1_n$ for  $n=n(\ell,d)$ large enough, see the proof of Proposition~\ref{prop:ex.reg}. The uniform $L^1_n$-control of non-negative solutions, see Lemma~\ref{l:moments}, is a consequence of the Fokker--Planck structure, which naturally ensures the propagation of higher moment bounds.

The canonical norm on $X:=\XX$ will often be abbreviated by $\|\cdot\|_X$, i.e.\ we let
$\|f\|_X:=\|f\|_\XX:=\max\{\|f\|_{L^\infty_\ell},\|f\|_{L^1_n}\}$.

\begin{lemma}[Local existence for~\eqref{eq:fpn},~\eqref{eq:fpnreg} in $X$ 
	and basic properties]\label{l:locex.mild}
	Let $\ell,n\ge1$. 
	For any $L\in(0,\infty)$ there exists $T=T(L)>0$ such that 
	for every $\fin \in \XX$ with $\|\fin\|_\XX\le L$
	there exists a unique mild solution 
	$f\in C([0,T];\XX)$ of equation~\eqref{eq:fpn}.
	
	On any time interval $[0,T^*)$, where the local-in-time mild solution exists, one has the extra regularity 
	$t\mapsto \nu(t)^\frac{1}{2}|\nabla f(t)|\in C_b((0,T); X_{\ell,1})$ for every $T\in(0,T^*)$.
	
	Furthermore, if $\fin\ge0$, the following additional properties hold true:	
	\smallskip
	
	\begin{enumerate}[label=\normalfont{(\roman*)}]
		\item\label{it:pos} Positivity: $f\ge0$ in $(0,T^*)\times\mathbb{R}^d$.
		\item\label{it:Cinfty} Smoothness: $f\in C^{1,2}((0,T^*)\times\mathbb{R}^d)$ and~\eqref{eq:fpn} holds in the classical sense.
		\item\label{it:mass} Mass conservation: $\|f(t)\|_{L^1}=\|\fin\|_{L^1}$ for all $t\in(0,T^*)$.
		\item\label{it:isotropy} Preservation of radial symmetry: if $\fin$ is isotropic, so is $f(t)$ for all $t\in(0,T^*)$.
	\end{enumerate}
	Given non-negative initial data $\fin^{(i)}\in \XX,$ $ i=1,2$, denote by $f^{(i)}, i=1,2$, the mild solution emanating from $\fin^{(i)}$ and let  $[0,T^*)$ be a common time interval of existence. Then
	\begin{enumerate}[label=\normalfont{(\roman*)},resume]
		\item\label{it:contr} $L^1$-Contractivity: $\|f^{(1)}(t)-f^{(2)}(t)\|_{L^1}\le \|\fin^{(1)}-\fin^{(2)}\|_{L^1}$ for all $t\in(0,T^*)$.
		\item\label{it:cp} Comparison: if $\fin^{(1)}\le \fin^{(2)}$, then
		$f^{(1)}\le f^{(2)}$ in $(0,T^*)\times\mathbb{R}^d$.
	\end{enumerate}
\medskip

\noindent Completely analogous statements hold for the regularised problem~\eqref{eq:fpnreg}.	
\end{lemma}

\begin{proof}
The proof is similar to that of~\cite[Theorem~2.5]{CLR_2009}. Abbreviate $X:=\XX$.
To prove the existence of a mild solution $f\in C([0,T];X)$ of~\eqref{eq:fpn}
we show that the operator
\begin{align}\label{eq:116T}
\mathcal{T}[f](t):=	\mathscr{F}[\fin](t) +
\int_0^t\ee^{-(t-s)}\int_{\mathbb{R}^d}
	\nabla_v\mathcal{F}(t{-}s,v,w)\cdot w\,|f|^{\gamma}f|_{(s,w)}\,\dd w\,\dd s
\end{align}
defines a contraction mapping on a closed ball in $C([0,T];X)$ provided $T=T(\|\fin\|_X)>0$ is small enough.
For this purpose, we rely on~\eqref{eq:FPsemigr} and estimate 
	\begin{align}
		\|\mathcal{T}[f](t)\|_{L^1_n}&\le C_T\|\fin\|_{L^1_n} 
		+ C_T\int_0^t\nu(t-s)^{-\frac{1}{2}}\||\cdot||f(s)|^{\gamma+1}\|_{L^1_n}\,\dd s,
	\end{align}
where $|\cdot||f(s)|^{\gamma+1}$ denotes the function $w\mapsto|w||f(s,w)|^{\gamma+1}$.
We next observe that
	\begin{align}
		\||\cdot||f|^{\gamma+1}\|_{L^1_n(\mathbb{R}^d)}&\le
		\int_{\mathbb{R}^d}(1+|w|^n)(1+|w|)|f(w)|^{\gamma+1}\,\dd w
		\\&\le \int_{\mathbb{R}^d}(1+|w|^n)|f(w)|\,\dd w
		\;\| (1+|\cdot|)|f|^\gamma\|_{L^\infty}
		\\&\le \|f\|_{L^1_n}\|f\|_{L^\infty_{1}}^\gamma,
	\end{align}
	where the last step uses the fact that $\gamma\ge1$.
	
	Next, we estimate for $p:=d+1$
	\begin{align}
		\|\mathcal{T}[f](t)\|_{L^\infty_\ell}&\le C_T\|\fin\|_{L^\infty_\ell} 
		+ C_T\int_0^t\nu(t-s)^{-\frac{1}{2}-\frac{d}{2p}}
		\||\cdot||f(s)|^{\gamma+1}\|_{L^p_\ell}\,\dd s
	\end{align}
	and 
	\begin{align}
		\||\cdot||f|^{\gamma+1}\|_{L^p_\ell(\mathbb{R}^d)}^p&\le
		\int_{\mathbb{R}^d}(1+|w|^{\ell})^{p}(1+|w|^p)|f(w)|^{(\gamma+1)p}\,\dd w
		\\&\le C\int_{\mathbb{R}^d}(1+|w|^{(\ell+1)p})\big(1+|w|^{\ell[(\gamma+1)p-1]}\big)^{-1}|f|\,\dd w
		\;\| (1+|\cdot|^\ell) f\|_{L^\infty}^{(\gamma+1)p-1}
		\\[-.3cm]&\le C\|f\|_{L^1_n}\| f\|_{L^\infty_\ell}^{(\gamma+1)p-1},
	\end{align}
where the last step uses the fact that
\begin{align}
	(\ell+1)p-\ell[(\gamma+1)p-1]=	\ell p+p-\ell p-\ell\gamma p+\ell 
	=\ell+d+1-\ell(d+1)\gamma\le 1\le n,
\end{align}
which follows from the choice $p=d+1$, $\ell\ge1$ and $\gamma\ge1$.

In combination, this shows that the mapping $\mathcal{T}$ obeys an estimate of the form
\begin{align}
		\|\mathcal{T}[f](t)\|_{X}\le C_T\|\fin\|_{X} +C_T\kappa(T)\|f\|_{C([0,T];X)}^{\gamma+1},\quad t\in[0,T],
\end{align}
for some function $\kappa\in C([0,\infty))$ that satisfies $\kappa(0)=0$.

Using the above estimates and analogous bounds for the difference $\mathcal{T}[f]-\mathcal{T}[\tilde f]$, one may now follow~\cite{CLR_2009} to show the contraction mapping property of $\mathcal{T}$ and deduce the existence of  a fixed point $f\in C([0,T];X)$ for small enough $T$ as asserted in Lemma~\ref{l:locex.mild}. By construction, this fixed point is a mild solution of~\eqref{eq:fpn}. 
 The extra regularity $t\mapsto \nu(t)^\frac{1}{2}|\nabla f(t)|\in C_b((0,T);(L^\infty_\ell\cap L^1_1)(\mathbb{R}^d))$ follows from similar arguments (see~\cite[Section~2.2]{CLR_2009}) combined with the uniqueness of mild solutions in 
$C([0,T];(L^\infty\cap L^p_1)(\mathbb{R}^d))$ for $p>d$ shown in Lemma~\ref{l:uniq.mild}.
(Of course, the contraction mapping property, whose proof we have not presented in full detail, also provides uniqueness.)

The properties~\ref{it:pos}--\ref{it:cp} can be deduced from classical arguments as in~\cite[Sections~2.3 and 2.4]{CLR_2009} (see also~\cite{LSU_1968,QS_2019}).
	
The analogous results for the regularised problem~\eqref{eq:fpnreg} are obtained along the same lines using in particular the bound $0\le\eta_\ve(f)\le|f|^\gamma$, where $\eta_\ve$ is defined in the line above~\eqref{eq:555}.
\end{proof}

\begin{lemma}[Uniform moment bound]\label{l:moments}
	Assume that $\ell\ge1$, $n\ge2$, and let $\fin\in \XX$ be non-negative. 
	Denote by $f_\ve\in C([0,T^*);\XX)$ the non-negative (local-in-time) mild solution of~\eqref{eq:fpnreg} as obtained in Lemma~\ref{l:locex.mild}.
	Then, for all $t\in[0,T^*)$, 
	\begin{align}\label{eq:204}
		\int_{\mathbb{R}^d}f_\ve(t,v)(1+|v|^n)\,\dd v \le C(n,d)\|\fin\|_{L^1_n}.
	\end{align}
\end{lemma}
We emphasize that the constant $C(n,d)<\infty$ in the above lemma is independent of $\ve$ and $T^*$. Moreover, the bound~\eqref{eq:204} equally holds for the local mild solution of~\eqref{eq:fpn}.
\begin{proof}
	Let us first provide the formal argument leading to the above estimate.
	We abbreviate $f:=f_\ve$ and define for $k\in[0,\infty)$
	\begin{align}
	 E_k(t) = \int_{\mathbb{R}^d}|v|^kf(t,v)\,\dd v.
	\end{align}
Then, by Lemma~\ref{l:locex.mild}, $E_0(t)\equiv \|\fin\|_{L^1(\mathbb{R}^d)}=:B_0$.
Clearly, $B_0\le \|\fin\|_{L^1_n}$.

We now argue inductively and assume that $\sup_{t\in[0,T^*)}E_{k-2}(t)\le B_{k-2}$ for some $k\in[2,n]$ and a positive constant $B_{k-2}$ obeying the bound $B_{k-2}\le C\|\fin\|_{L^1_n}$ with $C=C(n,d)$. 
Formally, we may then compute
\begin{align}\label{eq:mom.comp}
	\begin{aligned}
	\frac{1}{k}\frac{\dd}{\dd t}E_k(t)
	&= - \int_{\mathbb{R}^d}|v|^{k-2}v\cdot(\nabla f+vh_\ve(f))\,\dd v
	\\&= \int_{\mathbb{R}^d}\divv(|v|^{k-2}v) f\,\dd v
	 -\int_{\mathbb{R}^d}|v|^{k}h_\ve(f)\,\dd v
	 \\&\le (k-2+d)\int_{\mathbb{R}^d}|v|^{k-2} f\,\dd v
	  -\int_{\mathbb{R}^d}|v|^{k}f\,\dd v
	  \\& = (k-2+d)E_{k-2}(t) - E_k(t),
  \end{aligned}
\end{align}
which implies that 
\begin{align}\label{eq:Ek}
	E_k(t)\le \max\{E_k(0),(k-2+d)B_{k-2}\}=:B_k, \qquad t\in[0,T^*).
\end{align}
Since $k\le n$, the new upper bound $B_k$ again satisfies the estimate $B_{k}\le C\|\fin\|_{L^1_n}$ for some possibly larger constant $C=C(n,d)$.

We now let $k=2$ in the above step to find that $\sup_tE_2(t)\le \max\{E_2(0),dB_0\}=B_2$.
By interpolation we infer that $E_k(t)\le B_2^\frac{k}{2}B_0^\frac{2-k}{2}=:B_k$ for all $k\in(0,2)$ and all $t\in[0,T^*).$
Observe that $B_k\le C\|\fin\|_{L^1_n}$ for all $k\in(0,2)$.
We may now complete the induction argument:  starting with $k-2=n-2\floor*{n/2}$ (which lies in $[0,2)$) and iterating the above induction step $\floor*{n/2}$ times, we arrive at the bound $\sup_t E_n(t)\le C(n,d)\|\fin\|_{L^1_n}$.
	
Finally, let us note that the computation~\eqref{eq:mom.comp} can be made rigorous by introducing a smooth, compactly supported cut-off function $\vp_R$, $R\ge1$, with $\vp_R(v)=\vp(R^{-1}v)$ for some $\vp\in C^\infty_c(\mathbb{R}^d)$ satisfying $0\le\vp\le1$ and $\vp\equiv 1$ on $\{|v|\le 1\}$. The time derivative of 
$t\mapsto\frac{1}{k}\int|v|^kf(t,v)\vp_R(v)\,\dd v$ then satisfies an inequality which leads to~\eqref{eq:mom.comp}--\eqref{eq:Ek} in the limit $R\to\infty$.
\end{proof}

Global existence for~\eqref{eq:fpnreg} in $\XX$ will be obtained under the decay conditions
\begin{align}\label{eq:paramX.r}
	\begin{aligned}
		&	\ell\ge 1,
		\\& n\ge\ell+d+1.
	\end{aligned}
\end{align}
\begin{proposition}[Global existence for~\eqref{eq:fpnreg}]\label{prop:ex.reg}
	Let $\ve\in(0,1]$.
	Let $\ell,n$ satisfy~\eqref{eq:paramX.r}, and suppose that $\fin\in \XX$ is non-negative. 
	There exists a unique global-in-time mild solution $f_\ve\in C([0,\infty);\XX)$ of the Cauchy problem~\eqref{eq:fpnreg}. 	
\end{proposition}
Note that, as a consequence of Lemma~\ref{l:locex.mild}, the function $f_\ve$ in the above proposition enjoys the additional properties~\ref{it:pos}--\ref{it:isotropy}. In particular, it is a classical solution of~\eqref{eq:fpnreg} in $(0,\infty)\times\mathbb{R}^d$.

\begin{proof}
Local-in-time wellposedness of~\eqref{eq:fpnreg} in $X:=\XX$ follows from Lemma~\ref{l:locex.mild}.
Thus, for proving global existence it suffices to show that, for $\ve>0$ fixed, $\|f_\ve(t)\|_{X}$ cannot blow up in finite time. For this purpose, let $T<\infty$ and suppose that 
$f_\ve\in C([0,T);X)$ is a mild solution of~\eqref{eq:fpnreg} on the interval $[0,T)$.
Since $n\ge2$, we may invoke Lemma~\ref{l:moments} to infer that $\|f_\ve(t)\|_{L^1_n}$ remains bounded uniformly in time:
\begin{align}\label{eq:L1n.uni}
	\sup_{t\in[0,T)}\|f_\ve(t)\|_{L^1_n}\le C\|\fin\|_{L^1_n}<\infty.
\end{align}
Next, we let $p:=d+1$ and estimate for $t\in[0,T)$
\begin{align}\label{eq:Linfty.unif}
	\|f_\ve(t)\|_{L^\infty_\ell}&\le C_T\|\fin\|_{L^\infty_\ell} 
	+ C_T\int_0^t\nu(t-s)^{-\frac{1}{2}-\frac{d}{2p}}
	\|f_\ve\eta_\ve(f_\ve)\|_{L^p_{\ell+1}}\,\dd s,
\end{align}
where we used the fact that $\||\cdot|\tilde f(\cdot)\|_{L^p_{\ell}}\le 2\|\tilde f\|_{L^p_{\ell+1}}$ for $\tilde f\in L^p_{\ell+1}(\mathbb{R}^d)$.
Since $$ (\ell+1)p-\ell(p-1)= \ell+d+1\le n$$
and hence $(1+|w|^{\ell+1})^p\lesssim (1+|w|^n)(1+|w|^\ell)^{p-1}$,
 we further have
\begin{align}
	\|f_\ve\eta_\ve(f_\ve)\|_{L^p_{\ell+1}}^p&\le
	C(\ve)^p\int_{\mathbb{R}^d}(1+|w|^{\ell+1})^{p}|f_\ve|^p\,\dd w
\\&\le C(\ve)^p\,\|f_\ve\|_{L^1_n}\| f_\ve\|_{L^\infty_\ell}^{p-1}.
\end{align}
Hence, using the Young inequality $ab\le \frac{1}{p}a^p+\frac{p-1}{p}b^\frac{p}{p-1}$, we deduce
\begin{align}
	\|f_\ve\eta_\ve(f_\ve)\|_{L^p_{\ell+1}}&\le
	C(\ve)\,\|f_\ve\|_{L^1_n}+ C(\ve)\| f_\ve\|_{L^\infty_\ell}.
\end{align}
Inserting this bound into~\eqref{eq:Linfty.unif}, using~\eqref{eq:L1n.uni}, and applying the generalised Gronwall inequality~\cite[Theorem~3.3.1]{Amann_1995} yields
\begin{align}\label{eq:Linfty}
	\sup_{t\in[0,T)}\|f_\ve(t)\|_{L^\infty_\ell}&\le C_{T,\ve}(\|\fin\|_{L^\infty_\ell} +\|\fin\|_{L^1_n}) 
	\exp\left(C_{T,\ve}\right),
\end{align}
where the constants $C_{T,\ve}$ depend on $\ve,T$ and fixed parameters.
This shows that the unique local-in-time mild solution can be extended beyond the time $T$, and since $T\in(0,\infty)$ was arbitrary, the function $f_\ve$ extends to a unique global-in-time mild solution $f_\ve\in C([0,\infty);X)$. 
\end{proof}

For later reference, let us note the following consequence of the above theory.

\begin{corollary}[Short-time consistency]\label{cor:stc}
	Assume the hypotheses of Proposition~\ref{prop:ex.reg}.
	Let $f\in C([0,T];\XX)$ be a local-in-time mild solution of~\eqref{eq:fpn}, let $0<\epsilon_*<(\|f\|_{C([0,T];L^\infty)})^{-1}$ and $\ve\in(0,\epsilon_*]$.
	Then, since $h_\ve(s)=h(s)$ for $s\le \ve^{-1}$, the function $f$ is also
	the unique mild solution $f=f_\ve\in C([0,T];\XX)$ of~\eqref{eq:fpnreg} in $[0,T]$. In particular, as long as the mild solution of~\eqref{eq:fpn} obtained in Lemma~\ref{l:locex.mild} exists, the scheme $\{f_\ve\}_\ve$ trivially converges to this solution as $\ve\downarrow0$.
\end{corollary}

\subsection{Uniform bounds}\label{ssec:unif.bd}
\subsubsection{Preliminaries}\label{sssec:prelim.ub}
From now on we assume hypothesis~\ref{it:HP.init}, which imposes a somewhat stronger decay condition on the initial data as compared to Proposition~\ref{prop:ex.reg}. 
In particular, $\fin\in X_{\ell,\ell+d+1}$ with $\ell=d$ (resp.\ $\ell>3d+1$) if $\fin$ is isotropic 
(resp.\ anisotropic).
Let us note that the specific regularity conditions in~\ref{it:HP.init} have been made for convenience, and we have not attempted to optimise them.

Let $f\in C([0,T];X_{\ell,\ell+d+1})$ denote the local mild solution of~\eqref{eq:fpn} obtained in Lemma~\ref{l:locex.mild}.
Then, replacing $f$ by the time-shifted solution $f(t_0+\cdot)$ emanating from $f(t_0)$ for some small $t_0\in(0,T/2)$,
we may henceforth assume, without loss of generality, the additional regularity 
$f\in C^{1,2}([0,T/2]\times \mathbb{R}^d)$ with $\nabla f\in C([0,T/2];L^\infty_d(\mathbb{R}^d))$
and moreover, that $f$ is strictly positive in $[0,T/2]\times\mathbb{R}^d$
(if $m>0$, the strict positivity of $f(t_0)$ follows from~\cite[Proposition~52.7]{QS_2019}).

Thus, from now on we may assume the following stronger version of hypothesis~\ref{it:HP.init}:
\begin{align}\label{eq:hpinit}\tag{(H2')}
	\text{(H2')}\begin{cases}
&\text{\ref{it:HP.init} and }
	\text{the local regular solution $f$}\text{ of~\eqref{eq:fpn} with }f(0)=\fin 
	\\&\text{ satisfies } f\in C^{1,2}([0,\tau_*]\times \mathbb{R}^d),\;\nabla f\in C([0,\tau_*];L^\infty_d(\mathbb{R}^d)),
	 f>0\text{ in }[0,\tau_*]\times\mathbb{R}^d
	 \\&\text{ for some fixed $\tau_*>0$.}
	\end{cases}
\end{align}
Furthermore, we henceforth denote by $f_\ve,$ $\ve\in(0,\epsilon_*],$ the global mild solution of the regularised equation~\eqref{eq:fpnreg} as obtained in Proposition~\ref{prop:ex.reg}, where $\epsilon_*\in(0,1]$ is chosen small enough such that 
\begin{align}\label{eq:stc}
	f_\ve\equiv f\text{ in }[0,\tau_*]\;\;\text{ for all }\ve\in(0,\epsilon_*]. 
\end{align}
Such $\epsilon_*$ exists in virtue of Corollary~\ref{cor:stc}.

\subsubsection{Isotropic solutions}\label{sssec:isotropic}

In this subsection we assume $\fin$ to be isotropic and write $\gin(r)=\fin(v)$, $|v|=r$.
By Proposition~\ref{prop:ex.reg}~\ref{it:isotropy}, the global mild solution $f_\ve$ of~\eqref{eq:fpnreg} is isotropic, allowing us to write $g_\ve(t,r):=f_\ve(t,v)$ for $r=|v|\ge0$.
Observe that $g_\ve$ satisfies the equation
\begin{equation}\label{eq:114}
	\begin{aligned}
	\partial_tg_\ve &= r^{-(d-1)}\partial_r\Big(r^{d-1}\partial_rg_\ve+r^dh_\ve(g_\ve)\Big) \text{ in }\mathbb{R}_+\times\mathbb{R}_+,
	\\ 0&=\lim_{r\to0}\Big(r^{d-1}\partial_rg_\ve+r^dh_\ve(g_\ve)\Big), 
\end{aligned}
\end{equation}
where the limit in the last line holds locally uniformly in $t\in[0,\infty)$.

Our fundamental a priori bound for~\eqref{eq:fpn} relies on the fact that, in the isotropic case, equation~\eqref{eq:fpnreg} can be expressed as an evolution equation for the partial mass function
\begin{align}\label{eq:208}
	M_\ve(t,r):=\int_0^rg_\ve(t,\rho)\rho^{d-1}\dd\rho=c_d^{-1}\int_{B_r}f_\ve(t,v)\,\dd v 
	\le c_d^{-1}\|\fin\|_{L^1(\mathbb{R}^d)},
\end{align}
where $c_d$ denotes the area of the unit sphere $\partial B_1$ in $\mathbb{R}^d$. The equation for $M_\ve$ is obtained by multiplying~\eqref{eq:114} by $r^{d-1}$ and integrating in $r$
\begin{align}\label{eq:120}
	\partial_tM_\ve=r^{d-1}\partial_rg_\ve+r^dh_\ve(g_\ve).
\end{align}
Using  the relations
\begin{align*}
	\partial_rM_\ve&=r^{d-1}g_\ve,
	\\\partial_r^2M_\ve&=r^{d-1}\partial_rg_\ve+\tfrac{(d-1)}{r}\partial_rM_\ve,
\end{align*}
one arrives at
\begin{equation}\label{eq:Mreg}
	\begin{cases}
	\begin{aligned}
	\partial_tM_\ve&=\partial_r^2M_\ve-\tfrac{(d-1)}{r}\partial_rM_\ve+r^dh_\ve(r^{1-d}\partial_rM_\ve),
	\quad &&t>0,\;r\in\mathbb{R}_+,\\	
	M_\ve(t,0)&=0, &&t>0,
	\\ M_\ve(0,r)&=\Min(r), && r\in\mathbb{R}_+.
\end{aligned}
\end{cases}
\end{equation}
We note that, as a consequence of~\eqref{eq:stc}, 
\begin{align}\label{eq:206}
	M_\ve\equiv M\text{ in }[0,\tau_*]\times[0,\infty)\;\text{ for all }\ve\in(0,\epsilon_*],
\end{align}
where $M(t,r)=c_d^{-1}\int_{B_r}f(t,v)\,\dd v$ with $f\in C([0,\tau_*];X_{d,2d+1})$ 
denoting the local-in-time mild solution of~\eqref{eq:fpn}. 
Hence, thanks to the regularity established in Lemma~\ref{l:locex.mild} and hypothesis~\ref{eq:hpinit} we can ensure that
 \begin{align}\label{eq:205}
 M\in C^{1,2}([0,\tau_*]\times[0,\infty))\;	\text{ with }\; \sup_{\tau\in [0,\tau_*]}\|\partial_tM(\tau,\cdot)\|_{L^\infty([0,\infty))}\le K<\infty,\qquad
 \end{align}
where the last estimate follows from~\eqref{eq:120} and the regularity 
$f\in C([0,\tau_*];L^\infty_d(\mathbb{R}^d))$, 
$\nabla f\in C([0,\tau_*];L^\infty_d(\mathbb{R}^d))$.

\begin{proposition}[Global Lipschitz regularity in time]\label{prop:cpbound} 
	Suppose that $\fin$ is isotropic and satisfies the hypotheses in~\ref{eq:hpinit}.
	Denote by $M_\ve$ the partial mass function~\eqref{eq:208} of the global solution $f_\ve$ of~\eqref{eq:fpnreg} obtained in Proposition~\ref{prop:ex.reg}. In particular, $M_\ve$ is a classical solution of equation~\eqref{eq:Mreg} satisfying~\eqref{eq:206},~\eqref{eq:205} and is such that $f_\ve$ enjoys the uniform moment bound~\eqref{eq:204} for $n=2d+1$. 
Then 
	\begin{align}\label{eq:apriori}
		\sup_{\ve\in(0,\epsilon_*]}\sup_{t,r>0}|\partial_tM_\ve(t,r)|\le \kcp,
	\end{align}
where
\begin{align}\label{eq:kcp}
	\kcp:=\max\{K,\tfrac{\tilde m}{\tau_*}\}<\infty
\end{align}
 with $K$ as in~\eqref{eq:205} and $\tilde m:=c_d^{-1}m=c_d^{-1}\|\fin\|_{L^1(\mathbb{R}^d)}$. 
\end{proposition}

\begin{proof}
	Let $\kcp$ be as in~\eqref{eq:kcp}. We will show by contradiction that 
	\begin{align*}
		\sup_{\ve\in(0,\epsilon_*]} \big(M_\ve(t,r)-M_\ve(s,r)\big)\le \kcp|t-s|
	\end{align*}
 for all $t,s,r>0$. 
 
	Suppose the last inequality is false for some $\ve>0$. Then there exist $t_1,s_1,r_1>0$ such that 
	\begin{align*}
		M_\ve(t_1,r_1)-M_\ve(s_1,r_1)-\kcp|t_1-s_1|>0.
	\end{align*}
	Pick some $ T\ge\max\{t_1,s_1\}\,{+}\,1$. Without loss of generality we further assume that $T>\tau_*$ with $\tau_*$ being as in~\eqref{eq:206},~\eqref{eq:205}.
	Then, for $\spi>0$ small enough, we have 
	\begin{align*}
		M_\ve(t_1,r_1)-\frac{\spi}{T-t_1}-\frac{\spi}{T-s_1}-M_\ve(s_1,r_1)-\kcp|t_1-s_1|>0
	\end{align*}
	and hence
	\begin{align*}
		\sup_{(t,s,r)\in Q} \Big(
			M_\ve(t,r)- M_\ve(s,r)-\kcp|t-s|-\frac{\spi}{T-t}-\frac{\spi}{T-s}
		\Big)>0,
	\end{align*}
	where $Q=(0,T)\times(0,T)\times(0,\infty)$.
	
	We assert that the function 
	\begin{align}
		 P(t,s,r):=M_\ve(t,r)- M_\ve(s,r)-\kcp|t-s|-\frac{\spi}{T-t}-\frac{\spi}{T-s}
	\end{align}
attains its (positive) supremum in the interior of $Q$. 
	This can be seen as follows: by the uniform continuity of $M_\ve$ on $[0,T]\times[0,1]$ and the fact that $M_\ve(\cdot,0)\equiv 0$, there exists $r'>0$ such that  $P<0$ in $[0,T]\times[0,T]\times[0,r']$.  Moreover, by~\eqref{eq:206} and~\eqref{eq:205}
	one has $P<0$ in $[0,\tau_*]\times[0,\tau_*]\times[0,\infty)$.  The bound $M_\ve\le \tilde m$ further shows that $P<0$ in $[0,T]\times[T-\epsilon,T]\times[0,\infty)$ and in 
	$[T-\epsilon,T]\times[0,T]\times[0,\infty)$ for some $\epsilon=\epsilon(\spi, \tilde m)>0$.
	Next, for all $\bar s\in[\tau_*,T]$ and $r\in[0,\infty)$, we have $P(0,\bar s,r)\le \tilde m-\kcp\tau_*-\frac{2\delta}{T}<0$ thanks to the choice of $\kcp$. Likewise, $P(\bar t,0,r)\le -\frac{2\delta}{T}$ for all $\bar t\in[t^*,T]$ and $r\in[0,\infty)$. 
		Hence, it remains to rule out the existence of a maximising sequence $(t_n,s_n,r_n)$ with $r_n\to\infty$. To this end, we take advantage of the bound~\eqref{eq:204} (for $n=2$) to estimate
	\begin{align*}
		P(t,s,r)&\le c_d^{-1}\int_{\mathbb{R}^d\setminus B_r}f_\ve(s,v)\,\dd v
		-\frac{2\spi}{T}
		\\&\le \frac{1}{c_d (1+r)}\int_{\mathbb{R}^d\setminus B_r}f_\ve(s,v)
		\,(1+|v|)\,\dd v-\frac{2\spi}{T}
		\\&\le  \frac{1}{c_d (1+r)}\|\fin\|_{L^1_2(\mathbb{R}^d)}-\frac{2\spi}{T}.
	\end{align*}
	Observe that the right-hand side is negative whenever $r\ge R_*$ for a finite radius $R_*=R_*(\|\fin\|_{L^1_2},T,\delta)$ large enough. Hence, the same is true for $P(t,s,r)$.

	 Thus, the supremum of $P$ must be attained at some interior point 
	 $p^*=(t,s,r)\in Q$. At the point $p^*$ we have the optimality conditions 
	\begin{align*}
		&\partial_tM_\ve(t,r)-\kcp\frac{t-s}{|t-s|}=\frac{\spi}{(T-t)^2},
		\\& -\partial_s M_\ve(s,r)+\kcp\frac{t-s}{|t-s|}=\frac{\spi}{(T-s)^2},
	\end{align*}
	and hence 
	\begin{align*}
		\partial_t M_\ve(t,r)-\partial_s M_\ve(s,r)=\frac{\spi}{(T-t)^2}+\frac{\spi}{(T-s)^2}.
	\end{align*}
	Moreover, 
	\begin{align*}
		\partial_rM_\ve(t,r)=\partial_r M_\ve(s,r)
	\end{align*}
	and thus
	\begin{align}\label{eq:296}
		h_\ve(r^{1-d}\partial_rM_\ve(t,r))-h_\ve(r^{1-d}\partial_rM_\ve(s,r))=0.
	\end{align}
Further note that $0\ge \partial_r^2P(t,s,r)=\partial_r^2M_\ve(t,r)-\partial_r^2M_\ve(s,r)$.

	In combination with equation~\eqref{eq:Mreg} we deduce at the point $(t,s,r)=p^*:$
	\begin{align*}
		0=   \partial_t M_\ve(t,r)-\partial_s M_\ve(s,r)-(\partial_r^2M_\ve(t,r)-\partial_r^2M_\ve(s,r))&
		\\ \ge \frac{\spi}{(T-t)^2}+\frac{\spi}{(T-s)^2}&>0,
	\end{align*}
	which is a contradiction. This completes the proof of Proposition~\ref{prop:cpbound}.
	
	Let us remark that, thanks to the smoothness of $M_\ve$, estimate~\eqref{eq:apriori} may alternatively be proved by directly considering the equation satisfied by $N_\ve:=\partial_tM_\ve$, at least if one assumes a slightly stronger decay hypothesis on $\fin$. Indeed, notice that positive constants 
	above $\sup N_\ve(0,\cdot)$
	of the problem for $N_\ve$ are supersolutions, while negative constants below $\inf N_\ve(0,\cdot)$ are subsolutions. And if $f_\ve\in C([0,T];L^\infty_\ell(\mathbb{R}^d))$ and $\nabla f_\ve\in C([0,T];L^\infty_{\ell-1}(\mathbb{R}^d))$ for some $\ell>d$, we may use~\eqref{eq:120} to find that 
	$N_\ve(t,r)\to0$ as $r\to\infty$, uniformly in $t\in[0,T]$.
\end{proof}

The comparison principle underlying the proof of Proposition~\ref{prop:cpbound} can further be used to deduce monotonicity in $\ve$ of $M_\ve(t,r)$.

\begin{proposition}[Monotonicity of the scheme]\label{prop:mon.scheme}
	Let the hypotheses of Proposition~\ref{prop:cpbound} hold. 
	For any $0<\ve'\le\ve\le\epsilon_*$
	\begin{align}
		M_{\ve'}(t,r)\ge M_\ve(t,r),\qquad t,r>0.
	\end{align}
\end{proposition}
\begin{proof}
To begin with, we recall that $h_\ve\le 	h_{\ve'}$ whenever $0<\ve'\le\ve$ because of the non-increase of the function $(0,\infty)\ni s\mapsto s^{-\gamma}\eta(s)$.

	The remaining reasoning is similar to the proof of Proposition~\ref{prop:cpbound}.
	By contradiction, one assumes that there exist $t_1,r_1>0$ such that $M_\ve(t_1,r_1)-M_{\ve'}(t_1,r_1)$ is positive.
	Next, one fixes a finite time horizon $T\ge t_1+1$ and picks $\delta>0$ small enough such that 
	the function \[\tilde P(t,r):=M_\ve(t,r)-M_{\ve'}(t,r)-\tfrac{\delta}{T-t}\]
	has a positive supremum on $(0,T)\times(0,\infty)$.
	At an interior maximum point, one uses elementary calculus as before, where 
	the main difference is that instead of line~\eqref{eq:296}, we have now an inequality
	\begin{align}
		h_\ve(r^{1-d}\partial_rM_\ve(t,r))-h_{\ve'}(r^{1-d}\partial_rM_{\ve'}(t,r))\le 0.
	\end{align}
	The conclusion is then obtained by conceptually following the proof of Proposition~\ref{prop:cpbound}.
\end{proof}

The bound in Proposition~\ref{prop:cpbound} combined with the conservation of mass allows us to infer a uniform pointwise bound of the family $\{f_\ve\}_\ve$ away from the origin. Let us emphasize that, at this stage,  we do not aim for optimal blow-up rates as $r\downarrow0$. Such sharp rates will be derived in Section~\ref{sec:profile}.
\begin{lemma}[Bound away from origin: isotropic case]\label{l:boundiso}
	Assume the hypotheses of Proposition~\ref{prop:cpbound} and let $\kcp$ be as in~\eqref{eq:kcp}.
	Then for all $\ve\in(0,\epsilon_*]$, all $t>0$ and all $r>0$
\begin{align}\label{eq:apri>0}
	g_\ve(t,r)\le 2\max\{K_*,d\tilde m\}r^{-d},
\end{align}
	where as before we let $g_\ve(t,|v|):=f_\ve(t,v)$ for $f_\ve(t,\cdot)$ isotropic.
\end{lemma}

\begin{proof}
The inequality $s\le h_\ve(s)$ and~\eqref{eq:120} imply the
bound $g_\ve(t,r)\le r^{-d}K_ *-r^{-1}\partial_rg_\ve$. Hence,
\begin{align}
	g_\ve(t,r)\le K_*r^{-d} \text{ whenever }\partial_rg_\ve(t,r)\ge0.
\end{align} 
Suppose now that $\partial_rg_\ve(t,r)<0$ for some $t,r>0$.
If $\partial_rg_\ve(t,\cdot)\le 0$ on $[2^{-\frac{1}{d}}r,r]$, then
$$g_\ve(t,r)\frac{r^d}{2d} =g_\ve(t,r) \int_{2^{-\frac{1}{d}}r}^r\rho^{d-1}\,\dd\rho\le \int_{2^{-\frac{1}{d}}r}^rg_\ve(t,\rho)\rho^{d-1}\,\dd\rho\le\tilde m,$$ 
where the second step uses the monotonicity of $g_\ve(t,\cdot)$ on $[2^{-\frac{1}{d}}r,r]$ and the third step follows from mass conservation. 
Otherwise, there exists $r_0\in[2^{-\frac{1}{d}}r,r]$ such that $\partial_rg_\ve(t,\rho)<0$ for all $\rho\in(r_0,r]$ and $\partial_rg_\ve(t,r_0)=0$.
In this case, we estimate
$$g_\ve(t,r)\le g_\ve(t,r_0)\le K_*r_0^{-d}\le 2K_*r^{-d}.$$
In combination, this shows the bound
$g_\ve(t,r)\le 2\max\{K_*,d\tilde m\}r^{-d}$ for all $r>0$ and every $t>0$.		
\end{proof}

\subsubsection{Anisotropic case}\label{sssec:ani.ex}
For non-isotropic initial data $\fin$ satisfying~\ref{eq:hpinit} and thus in particular 
 $\fin\in L^\infty_\ell(\mathbb{R}^d)$ for some $\ell>3d+1$, we consider as in~\cite{CCLR_2016} an isotropic 
envelope $\hat{f}_\mathrm{in}(v)\ge \fin(v)$ given by
\begin{align}
	\hat{f}_\mathrm{in}(v)=\frac{\|\fin\|_{L^\infty_\ell}}{(1+|v|^{\ell})}.
\end{align}
Since $\ell-(2d+1)>d$, the isotropic function $\hat{f}_\mathrm{in}$ satisfies~\ref{eq:hpinit}  and thus in particular the hypotheses of Proposition~\ref{prop:ex.reg}. 
Invoking this proposition, we obtain non-negative global-in-time (mild) solutions $f_\ve$ and $\hat f_\ve$ of~\eqref{eq:fpnreg} emanating from $\fin$ resp.\ $\hat{f}_\mathrm{in}$, where by the comparison property,
 Proposition~\ref{prop:ex.reg}~\ref{it:cp}, $f_\ve\le\hat f_\ve$ in $[0,\infty)\times \mathbb{R}^d$.
Thus, the uniform bound away from zero in the isotropic case (cf.\ Lemma~\ref{l:boundiso}) implies a similar result for anisotropic solutions:
\begin{corollary}[Bound away from origin: anisotropic case]\label{cor:boundaniso}
	Assume~\ref{eq:hpinit}, thus in particular $\fin\in L^\infty_\ell(\mathbb{R}^d)$ for some $\ell>3d+1$ if $\fin$ is non-isotropic.
There exists a finite (non-explicit) constant $\hat \kcp$ only depending on $\|\fin\|_{L^\infty_\ell}$, $\ell$ and fixed parameters such that for all $t>0$ and all $v\in\mathbb{R}^d\setminus\{0\}$
	\begin{align}\label{eq:bdAni}
		\qquad 	f_\ve(t,v)\le 2\max\big\{ \hat \kcp, d\hat m   \big\}|v|^{-d},
	\end{align}
where $\hat m=c_d^{-1}\|\hat\fin\|_{L^1}$.
\end{corollary}

\subsection{Passage to the limit}\label{ssec:passlim}

\begin{proof}[Proof of Proposition~\ref{prop:lim}]
	For $\ve\in(0,\epsilon_*]$ let $f_\ve$ be the global-in-time mild solution of~\eqref{eq:fpnreg}
	emanating from $\fin$ as constructed in Proposition~\ref{prop:ex.reg}.
	In the rest of this proof we abbreviate $\RdnoO:=\mathbb{R}^d\setminus\{0\}$.

We first show assertions~\ref{it:curve}--\ref{it:approxProp}.
	
		\paragraph{- \em{Approximation property~\ref{it:approxProp} and regularity of $f$}}\label{page:class}
We	assert that for every compact set $G\subset\subset(0,\infty)\times \RdnoO$, we have an $\ve$-uniform bound of the form
	\begin{align}\label{eq:597}
		\|f_\ve\|_{H^{1+\frac\alpha2,2+\alpha}(G)}\le C_G 
	\end{align}
for some $\alpha\in(0,1)$, where $H^{1+\frac\alpha2,2+\alpha}(G)$ denotes the parabolic H\"older space with $\tfrac{\alpha}{2}$-H\"older continuous first order temporal and $\alpha$-H\"older continuous second order spatial derivatives. 
	Inequality~\eqref{eq:597} can be shown using standard results on parabolic 
	regularity~\cite{LSU_1968,Lieberman_1996}. 
	To sketch the main points, we first observe that each $f_\ve$
	is strictly positive unless $\fin\equiv0$ (cf.~\cite[Prop.~52.7]{QS_2019}) and smooth in $(0,\infty)\times\mathbb{R}^d$. Moreover, as a consequence of Lemma~\ref{l:boundiso} resp.\ Corollary~\ref{cor:boundaniso}, the family 
	$\{f_\ve\}_\ve$ is $\ve$-uniformly bounded in $L^\infty(G)$.
	Hence, rewriting~\eqref{eq:fpnreg} as
	$\partial_tf_\ve=\Delta f_\ve+h_\ve'(f_\ve)v\cdot\nabla f_\ve+dh_\ve(f_\ve),$
	Theorem~11.1 in~\cite[Chapter~III]{LSU_1968} on linear parabolic equations provides us with an $\ve$-uniform gradient bound $\|\nabla f_\ve\|_{C^0(G)}\le C_G$. For higher-order spatial derivatives, $\ve$-uniform bounds on $G$ are obtained by applying a similar reasoning to the equation satisfied by $\partial_{v_i}f_\ve$ etc., and time regularity follows from the equation itself. 
	
	Hence, by the Arzel\`a--Ascoli theorem,
	 there exists a function $f\in C^{1,2}((0,\infty)\times U)$, $f\ge0$, such that,
	upon passing to a subsequence $\ve\downarrow0$ (not relabelled), 
	\begin{align}\label{eq:295}
	\hspace{2cm}	f_\ve\to f\quad \text{ in  }\;C^{1,2}(G)\;
		\text{ for every }G\subset\subset(0,\infty)\times U,
	\end{align}
and $f$ is a classical solution of~\eqref{eq:fpn} in $(0,\infty)\times U$.

Combining~\eqref{eq:295} with the  moment bound in Lemma~\ref{l:moments} yields, for all $\rho>0$, 
\begin{align}\label{eq:295.}
\lim_{\ve\to0}\|f_\ve(t)-f(t)\|_{L^1(\mathbb{R}^d\setminus B_\rho(0))}=0
\end{align}
 locally uniformly in $t\in[0,\infty)$. 
 Moreover, Fatou's lemma implies $\int_{\mathbb{R}^d}f(t)\le m$ for all $t$.
 
 Let us now show that $f$ is strictly positive for non-trivial initial data $\fin$, i.e.\ whenever $m>0$.
 For this purpose, we pick some $\theta>0$, define $\fin^\#:=\min\{f_{\infty,\theta},\fin\}$, and let $\{f_\ve^\#\}_{\ve\in(0,\epsilon_\theta]}$ with $\epsilon_\theta:=(\|f_{\infty,\theta}\|_{L^\infty(\mathbb{R}^d)})^{-1}>0$ denote the family of global-in-time mild solutions of \eqref{eq:fpnreg} starting from 
 $\fin^\#$. For $\ve\in(0,\epsilon_\theta]$ the steady state $f_{\infty,\theta}$ is a classical solution of \eqref{eq:fpnreg} with rapid decay as $|v|\to\infty$, and thus in particular a mild solution. Hence, the comparison principle in Lemma~\ref{l:locex.mild}~\ref{it:cp} implies that 
 \begin{align}\label{eq:cp.pos}
 	f_\ve^\#\le \min\{f_{\infty,\theta},f_\ve\}
 \end{align}
 showing in particular that $f^\#:=f_\ve^\#$ is independent of $\ve$ for $\ve\in(0,\epsilon_\theta]$ (cf.\ the argument in Cor.~\ref{cor:stc}). By Lemma~\ref{l:locex.mild}~\ref{it:Cinfty}, $f^\#$ is a non-negative classical solution of \eqref{eq:fpnreg} (and \eqref{eq:fpn}) with initial datum $\fin^\#\not\equiv 0$. From a classical strong comparison principle (see e.g.\ \cite[Prop.~52.7]{QS_2019}), comparing $f^\#$ with the zero solution, we deduce that $f^\#$ is strictly positive in $(0,\infty)\times\mathbb{R}^d$. 
 Taking the limit $\ve\to0$ in~\eqref{eq:cp.pos} along the subsequence obtained in~\eqref{eq:295}
 yields $f^\#\le f$, and thus provides us with a locally uniform positive lower bound for $f$ away from zero. 

The family of measures $\{\mu^{(\ve)}\}_\ve$, $\mu^{(\ve)}:=f_\ve\tvmeas_+$, is tight on any finite time horizon as ensured by Lemma~\ref{l:moments}. Hence, by Prokhorov's theorem, there exists a non-negative Radon measure $\mu$ on $[0,\infty)\times\mathbb{R}^d$ 
such that after possibly passing to another subsequence $\ve\downarrow0$
\begin{align}\label{eq:294}
	\mu^{(\ve)}\overset{\ast}{\rightharpoonup} \mu\text{ in }\mathcal{M}_+([0,T]\times\mathbb{R}^d) 
\end{align}
for any $T<\infty$. In fact, due to~\eqref{eq:295},~\eqref{eq:295.} and 
  $\int_{\mathbb{R}^d} f_\ve(t)\equiv m$, 
the passage to a subsequence $\ve\downarrow0$ would not have been necessary at this point (see also the next paragraph).

\paragraph{- \em{Mass-conserving curve and decomposition}}
By~\eqref{eq:295.} the family $\mu^{(\ve)}_t:=f_\ve(t)\mathcal{L}^d$ satisfies
\begin{align}\label{eq:mut}
	\lim_{\ve\to0}\int_{\mathbb{R}^d}\psi\,\dd\mu^{(\ve)}_t=\int_{\mathbb{R}^d}\psi(v)\,f(t,v)\dd v
\end{align}
 for all $\psi\in C_b(\mathbb{R}^d)$ with $\supp\psi\subset\mathbb{R}^d\setminus\{0\}$, where the limit is taken along the same sequence $\ve\downarrow0$ as in~\eqref{eq:295.}.
At the same time, the tightness of the family $\{\mu^{(\ve)}_t\}_{\ve}$ ensures, upon passing to a subsequence $\ve_j\downarrow0$ which may (initially) depend on $t$, that 
$\mu^{(\ve_j)}_t\weakstar \mu_t$ in $\mathcal{M}_+(\mathbb{R}^d)$ for some $\mu_t\in\mathcal{M}_+(\mathbb{R}^d)$ with $\mu_t(\mathbb{R}^d)=m$. As a consequence of~\eqref{eq:mut}, we have 
$\supp(\mu_t-f(t)\mathcal{L}^d)\subset\{0\}$, independent of the chosen subsequence $\ve_j\downarrow0$. 
Since $\mu_t(\mathbb{R}^d)=m$, this entails that
 $\mu_t(\{0\})=m-\|f(t)\|_{L^1(\mathbb{R}^d)}:=a(t)$. Hence, 
\begin{align}\label{eq:295..}
	\mu_t=\ptm(t)\delta_0+f(t)\mathcal{L}^d,\quad t\ge0,
\end{align}
and the convergence
\begin{align}\label{eq:300}
	\mu^{(\ve)}_t\weakstar \mu_t\quad \text{ in }\mathcal{M}_+(\mathbb{R}^d)
\end{align}
holds for the entire sequence $\ve\downarrow0$ as obtained in~\eqref{eq:295}--\eqref{eq:294}. 
Let now $\vp\in C_c([0,\infty)\times\mathbb{R}^d)$.
 On the one hand, identity~\eqref{eq:294} implies that
$\lim_{\ve\downarrow0}\int_{[0,\infty)\times\mathbb{R}^d}\vp\,\dd\mu^{(\ve)}=\int_{[0,\infty)\times\mathbb{R}^d}\vp\,\dd\mu$. On the other hand, the function $\iota_\ve(t):=\int_{\mathbb{R}^d}\vp(t,v)\,\dd\mu_t^{(\ve)}(v)$ admits the uniform bound $|\iota_\ve(t)|\le m\|\vp\|_{L^\infty}$ for all $t\ge0$ and converges pointwise 
to $\int_{\mathbb{R}^d}\vp(t,v)\,\dd\mu_t(v)$ as $\ve\downarrow0$.
Hence, using dominated convergence for the right-hand side (in conjunction with the compact support of $\vp$ in time), we may pass to the limit $\ve\downarrow0$ in the identity $$\int_{[0,\infty)\times\mathbb{R}^d}\vp\,\dd\mu^{(\ve)}=\int_{[0,\infty)}\int_{\mathbb{R}^d}\vp(t,v)\,\dd\mu_t^{(\ve)}(v)\dd t$$
to deduce the representation $\dd\mu=\dd\mu_t\dd t$.

To prove the asserted weak-$\ast$ continuity of the mapping $[0,\infty)\ni t\mapsto\mu_t\in\mathcal{M}_+(\mathbb{R}^d)$, 
let us first recall that $\mathcal{M}_+(\mathbb{R}^d)$ endowed with the weak-$*$ topology is metrizable 
(see e.g.~\cite{AGS_2008,Klenke_2014}), so that it suffices to show sequential continuity:
 given $\hat t\ge0$ and a sequence $(t_j)$ satisfying $\lim_{j\to\infty}t_j=\hat t$, we need to prove that $\mu_{t_j}\weakstar\mu_{\hat t}$ in $\mathcal{M}_+(\mathbb{R}^d)$.
By the Portmanteau theorem (see e.g.~\cite[Theorem 13.16]{Klenke_2014}), it suffices 
to show for every open subset $O\subset\mathbb{R}^d$ the estimate
\begin{align}\label{eq:wstar}
	\int_O\dd\mu_{\hat t}\le\liminf_{j\to\infty}\int_O\dd\mu_{t_j}.
\end{align}
If $0\in O$, then $B_\rho(0)\subset O$ for $\rho>0$ small enough, and~\eqref{eq:wstar} holds
with an equality. Indeed, the smoothness of $f$ away from $v=0$ and the moment bound in Lemma~\ref{l:moments} (which implies an analogous bound for the pointwise limit $f$ of $f_\ve$) ensure that $f(t_j)\to f(\hat t)$ in $L^1(\mathbb{R}^d\setminus O)$, and hence 
\begin{align}
	\int_O\dd\mu_{\hat t} = m-\int_{\mathbb{R}^d\setminus O}f(\hat t) 
	= \lim_{j\to\infty}\bigg(m-\int_{\mathbb{R}^d\setminus O}f(t_j) \bigg) 
	= \lim_{j\to\infty} \int_O\dd\mu_{t_j}.
\end{align}
 If $0\not\in O$,  inequality~\eqref{eq:wstar} is equivalent to 
$\int_Of(\hat t,v)\,\dd v\le\liminf_{j\to\infty}\int_Of(t_j,v)\,\dd v$ (in virtue of~\eqref{eq:295..}), and this bound is a consequence of Fatou's lemma since $f(t_j)\to f(\hat t)$ a.e.~in $\mathbb{R}^d$. This establishes~\eqref{eq:wstar}.

\smallskip
It remains to prove assertions~\ref{it:uniq.sola} and~\ref{it:lip.ptm}. 
\smallskip
	\paragraph{- \em{Unique limit}} We now show~\ref{it:uniq.sola}.
In the isotropic case,  Proposition~\ref{prop:mon.scheme} ensures that the limit $M(t,r):=\lim_{\ve\to0}M_\ve(t,r)=c_d^{-1}\lim_{\ve\to0}\mu_t^{(\ve)}(B_r)$ is well-defined for all $t,r>0$. Thus, in this case, the limiting density $f$ in~\eqref{eq:295} and hence $\mu$ can be uniquely recovered from $M$, which is independent of the choice of the sequence $\ve\downarrow0$.
In view of the above compactness properties, this implies assertion~\ref{it:uniq.sola}.
	
	\paragraph{- \em{Lipschitz continuity of point mass}}
	Restricting to isotropic data, we have for $r>0$,
	\begin{align}
		c_dM_\ve(t,r)=\mu^{(\ve)}_t(B_r)&\to\mu_{t}(B_r)&&\text{ as }\ve\to0,
		\\ \mu_t(B_r)&\to a(t)&&\text{ as }r\to0,
	\end{align}
	where the first line follows from~\eqref{eq:300} and the fact that 
	$\supp\mu_t^\mathrm{sing}\subseteq\{0\}$.
	Thus, the Lipschitz bound~\eqref{eq:apriori} implies that $|\ptm(t)-\ptm(s)|\le c_d\kcp|t-s|$, hence part~\ref{it:lip.ptm}.
\end{proof}

\section{Universal space profile}\label{sec:profile}

Equipped with the uniform control~\eqref{eq:apriori}, we will now combine ODE and bootstrap arguments 
with localised semi-group estimates to study the regularity and the space profile of the density $f$ near the origin.  A rigorous analysis is achieved by working with the family of
approximate solutions $f_\ve$ constructed in Section~\ref{ssec:mild}.
We will show that for isotropic data the solution at any fixed positive time is either regular and smooth, or the density of the regular part follows, up to a lower order term with explicit rate, a universal profile at the origin that is uniquely determined by the limiting steady state $f_c$. 
This even slightly improves the profile obtained in~\cite{CHR_2020} for $d=1$.

Throughout this section we assume~\ref{eq:hpinit} and let $\kcp$ denote the least upper bound such that inequality~\eqref{eq:apriori} holds true, that is
		\begin{align}\label{eq:defK}
			\kcp:=\sup_{\ve\in(0,\epsilon_*]}\sup_{t>0,r>0}|\partial_tM_\ve(t,r)|.
		\end{align}
In virtue of Section~\ref{sssec:prelim.ub}, it is clear that the main conclusion of the present section,   Theorem~\ref{thm:profile.r}, only requires hypothesis \ref{it:HP.init} and not its stronger version~\ref{eq:hpinit}.
	
\subsection{Lower and upper bounds}
	
The analysis in this subsection mostly concerns isotropic solutions, for which the uniform bound~\eqref{eq:apriori} is available.
As introduced in Section~\ref{sssec:isotropic}, in the isotropic case we write $g_\ve(t,r):=f_\ve(t,v)$ whenever $r=|v|>0$, and likewise $g(t,r):=f(t,v)$ for the pointwise limit obtained upon sending $\ve\downarrow0$.

		\begin{proposition}[Lower bound]\label{prop:profile.lb}
			Abbreviate $\alpha_c=\frac{2}{\gamma}$.
			In addition to~\ref{it:HPspc}, \ref{eq:hpinit} suppose that 
			\begin{align}
				\alpha_c+2-d>0.
			\end{align}
			Further assume that the initial value $f_\mathrm{in}$ is isotropic.
			Pick any $\ul\alpha\in((d-2)_+,\alpha_c)$. 
			 If $d=1$, assume in addition that\footnote{In dimension $d=1$ condition~\ref{it:HPspc} reduces to $\gamma>2$, which is implies that $\alpha_c>\frac{1}{\gamma-1}$.} $\ul\alpha>\frac{1}{\gamma-1}$.
		For $\alpha\in[\ul\alpha,\alpha_c]$ let $\tilde g(r)=c_\gamma r^{-\alpha}$,
		 where $c_\gamma=\big(2/\gamma\big)^{1/\gamma}$. For $t>0$ and $\ve\in(0,\epsilon_*]$ define
		\begin{align*}
			\tilde r_\ve=\tilde r_\ve(t)=\sup\{r>0: g_\ve(t,\rho)<\tilde g(\rho)\text{ for all }\rho\in(0,r)\}.
		\end{align*}	
		There exists a constant $B<\infty$ and
		a radius $r_*\in(0,1]$ only depending on $\kcp$ (cf.~\eqref{eq:defK}) and on $\gamma,d,\ul\alpha$ (but not on $\alpha$)
		such that for all $t>0$ and all $\ve\in(0,\epsilon_*]$ the following holds: whenever $\tilde r_\ve(t)\in(0, r_*)$, then
		\begin{align*}
			\qquad g_\ve(t,r)\ge \tilde g(r) -Br^{2-d} \qquad \text{ for } \; r\in(\tilde r_\ve(t),r_*).
		\end{align*}
	\end{proposition}

\begin{remark}\label{rem:lb}
Let us note that for $\alpha=\alpha_c$ it is (a priori) not clear whether
the unboundedness of the limiting function $f$ at some time $t$, i.e.\ $\|f(t)\|_{L^\infty(\mathbb{R}^d)}=\infty$, implies that $\liminf_{\ve\downarrow0}\tilde r_\ve(t)=0$. This is the main reason why the derivation of the universal lower bound 
 on the spatial singularity profile in Theorem~\ref{thm:profile.r} requires several further steps (cf.\ Sec.~\ref{ssec:inst.reg} and Sec.~\ref{ssec:profile.proof}). 
 For ruling out fine spike-like singularities in $f(t)$ near the origin that are dominated by a subcritical power law, i.e.\ by $Cr^{-\alpha}$ for some $\alpha<\alpha_c$, 
 we exploit the fact that such subcritical singularities are smoothed out instantaneously 
 (cf.\ Prop.~\ref{prop:ins.subcr}) and so cannot form at a positive time. 
 For dealing with potential intermediate situations (e.g.\ oscillatory power laws), it is crucial that the stability result in Proposition~\ref{prop:profile.lb} is valid not only for $\alpha=\alpha_c$ but also for a small range of subcritical exponents $\alpha$ near $\alpha_c$, see Section~\ref{ssec:profile.proof} for details.
\end{remark}

	\begin{proof}[Proof of Proposition~\ref{prop:profile.lb}]
			To begin with, we note that for any $\alpha\in[\ul\alpha,\alpha_c]$
		\begin{align}\label{eq:121}
			-\alpha\le-\ul\alpha<2-d\le4-d-\alpha\gamma.
		\end{align}		
				Let now $t>0$ and  $\ve\in(0,\epsilon_*]$.  Observe that the radius $\tilde r_\ve(t)$ may be infinite and that the  assertion of Proposition~\ref{prop:profile.lb} only concerns the case where $\tilde r_\ve(t)>0$ is small. (If $\tilde r_\ve(t)\ge1$ for all $\ve\in(0,\epsilon_*]$ and $t>0$,  the assertion is trivially satisfied for $r_*=1$.) Hence, in the following we may assume that $\tilde r_\ve(t)\in(0,1)$.
Then, by continuity, $g_\ve(t,\tilde r_\ve(t))=\tilde g(\tilde r_\ve(t))$, and we may define a radius $\tilde r_{1,\ve}>\tilde r_\ve$ via
\begin{align*}
	\tilde r_{1,\ve}(t):=\sup\{r\in(\tilde r_\ve(t),1):g_\ve(t,\rho)\ge \tfrac{1}{2}\tilde g(\rho)\text{ for all }\rho\in(\tilde r_\ve(t),r)\}.
\end{align*}

To proceed, we abbreviate $b_\ve(t,r):=\partial_tM_\ve(t,r)$ and note that (cf.~\eqref{eq:120})
	\begin{align*}
		r^{d-1}\partial_rg_\ve+r^dh_\ve(g_\ve)=b_\ve.
	\end{align*}
In the rest of the proof we are concerned with suitably estimating an integrated version
of this differential equation.
The following calculations being of a purely spatial type, we henceforth omit the fixed time argument $t$.
Recall that $\Phi'$ is a primitive of $\frac{1}{h}$, i.e.\ $\Phi''=\frac{1}{h}$ (cf.\ page~\pageref{eq:gradflow}), whence
\begin{align}
	\frac{\dd}{\dd r}\Phi'(g_\ve) =\frac{\partial_rg_\ve}{h(g_\ve)} 
	&= -r\frac{h_\ve(g_\ve)}{h(g_\ve)} + b_\ve r^{1-d}\frac{1}{h(g_\ve)}
	\ge -r -  K_* r^{1-d}g_\ve^{-(\gamma+1)}, 
\end{align}
where we used the fact that $\frac{h_\ve(s)}{h(s)}\le 1$ and $h(s)\ge s^{\gamma+1}$ for all $s>0$.
Renaming $r$ by $\rho$ and integrating the above inequality in space over $\rho\in(\tilde r_\ve,r)$ for $r\in(\tilde r_\ve,1]$ yields
\begin{align}\label{eq:444}
	\Phi'(g_\ve(r))-\Phi'(g_\ve(\tilde r_\ve))\ge -\frac{1}{2}r^2+\frac{1}{2}\tilde r_\ve^2
	-K_*\int_{\tilde r_\ve}^r\rho^{1-d}g_\ve(\rho)^{-(\gamma+1)}\,\dd\rho.
\end{align}
We next expand for $s\gg1$
\begin{align}\label{eq:Phi'}
	\Phi'(s)=-\frac{1}{\gamma}\log(s^{-\gamma}+1)
	=-\frac{1}{\gamma}s^{-\gamma}+O(s^{-2\gamma})
	=-\frac{1}{\gamma}s^{-\gamma}\big(1+O(s^{-\gamma})\big).
\end{align}
Note that the increasing function $\Phi':(0,\infty)\to(-\infty,0)$ is bijective and
 for $-1\ll\hat s<0$ 
\begin{align}\label{eq:Phi'.inv}
	(\Phi')^{-1}(\hat s) = \big(\exp(-\gamma\hat s)-1\big)^{-\frac{1}{\gamma}}= (-\gamma\hat s)^{-\frac{1}{\gamma}}\big(1+O(\hat s)\big).
\end{align}
Furthermore, we assert that there exists $r_\circ=r_\circ(\ul\alpha,\gamma)\in(0,\mathrm{e}^{-1}]$ such that 
\begin{align}\label{eq:446}
	r^2-\rho^2\le r^{\alpha\gamma}-\rho^{\alpha\gamma}  \qquad \text{for all }\; 0<\rho<r\le r_\circ.
\end{align}
 Inequality~\eqref{eq:446} can be shown as follows: since
$\beta:=\frac{\alpha\gamma}{2}\le 1$, concavity yields
\begin{align}\label{eq:447}
	\beta r^{\beta-1}(r-\rho) \le r^\beta-\rho^\beta \qquad \text{ for all}\;0<\rho<r<1.
\end{align}
Upon multiplication by $r+\rho\le r^\beta+\rho^\beta$, we deduce 
	$\beta r^{\beta-1}(r^2-\rho^2) \le r^{2\beta}-\rho^{2\beta}$. 
This implies~\eqref{eq:446}, since 
	 $\beta r^{\beta-1}\ge 1$ for all $\beta\in[\frac{\ul\alpha\gamma}{2},1]$ and $r\in(0,r_\circ]$ if  $r_\circ>0$ is small enough as above.
(To see the latter, note that $\beta r^{\beta-1}\ge 1$ is equivalent to $r\le \beta^{\frac{1}{1-\beta}}$, where $\beta^{\frac{1}{1-\beta}}\uparrow\mathrm{e}^{-1}$ as $\beta\uparrow1$.)

Letting $\rho=\tilde r_\ve$ in~\eqref{eq:446} we infer from~\eqref{eq:444}, using also~\eqref{eq:Phi'} and the identity
$g_\ve(\tilde r_\ve)=c_\gamma \tilde r_\ve^{-\alpha}$, 
\begin{align}
	\Phi'(g_\ve(r))&\ge 
	-\frac{1}{2}r^{\alpha\gamma}+O(r^{2\alpha\gamma}) -K_*\int_{\tilde r_\ve}^r\rho^{1-d}g_\ve(\rho)^{-(\gamma+1)}\,\dd\rho
\end{align}
whenever $\tilde r_\ve<r<r_\circ$.
For $\rho\in (\tilde r_\ve,\tilde r_{1,\ve})$ we have $g_\ve(\rho)^{-(\gamma+1)}\le 2^{\gamma+1}\tilde g(\rho)^{-(\gamma+1)}=: C_1(\gamma)\rho^{\alpha\gamma+\alpha}$. 
We will now show that there exists $r_*\in(0,r_\circ]$ only depending on $K_*$, $\underline\alpha$ and fixed parameters such that $\tilde r_{1,\ve}(t)\ge r_*$ for all $t>0$ and $\ve\in(0,\epsilon_*]$ for which $\tilde r_\ve(t)\in(0,1)$.
To this end, we let $\tilde r_{2,\ve}:=\min\{\tilde r_{1,\ve},r_*\}$ for some $r_*\in(0,r_\circ]$ to be fixed later.
For $r\in (\tilde r_\ve,\tilde r_{2,\ve}]$ we have 
\begin{align}
	\Phi'(g_\ve(r))&\ge 
	-\frac{1}{2}r^{\alpha\gamma}\big(1+K_*C_1(\gamma)r^{2-d+\alpha}+O(r^{\alpha\gamma}) \big).
\end{align}
The last two terms in the brackets on the right-hand side behave like $O(r^{2-d+\alpha})$ for $0<r\ll1$, because $2-d+\alpha< \alpha\gamma$ for all $\alpha\in[\ul\alpha,\alpha_c]$ (if $d=1$ it follows from the choice $\ul\alpha>\frac{1}{\gamma-1}$; if $d\ge2$ this follows from the condition $\gamma>\frac{2}{d}$ in~\ref{it:HPspc}, which implies that 
$\alpha(\gamma-1)>\alpha(\frac{2}{d}-1)=\frac{\alpha}{d}(2-d)\ge(2-d)$ since $2-d\le0$ and $\alpha\le\frac{2}{\gamma}<d$), 
where the hidden constants in $O(\cdot)$ only depend on $K_*$ and fixed parameters. 
Hence, using the fact that $\Phi'$ is increasing and recalling the expansion~\eqref{eq:Phi'.inv}, we infer for  $r\in(\tilde r_\ve,\tilde r_{2,\ve}]$ 
\begin{align}\label{eq:448}
	\begin{aligned}
g_\ve(r)&\ge 
	(\Phi')^{-1}(-\tfrac{1}{2}r^{\alpha\gamma}\big(1 +O(r^{2-d+\alpha})\big))
	\\&=c_\gamma r^{-\alpha}\big(1+O(r^{2-d+\alpha})\big)
	=c_\gamma r^{-\alpha}+O(r^{2-d}).
\end{aligned}
\end{align}
Since $2-d+\underline\alpha>0$, this shows that after possibly decreasing $r_*\in(0,r_\circ]$ 
(only depending on $K_*$, $\underline\alpha$ and fixed parameters) we can ensure that 
$g_\ve(t,r)\ge \frac{3}{4}\tilde g(r)$ for all $r\in(\tilde r_\ve,\tilde r_{2,\ve}]$, $t>0$ and $\ve\in(0,\epsilon_*]$. As a consequence, $\tilde r_{1,\ve}(t)> \tilde r_{2,\ve}(t)$
and $\tilde r_{2,\ve}(t)=r_*$.
This, in turn, means that inequality~\eqref{eq:448} is valid for all $r\in(\tilde r_\ve,r_*]$ whenever $\tilde r_\ve(t)< 1$, completing the proof of Proposition~\ref{prop:profile.lb}. 
\end{proof}

	\begin{proposition}[Upper bound]\label{prop:upperBd}
		Use the notations and assume the hypotheses of Proposition~\ref{prop:profile.lb}.
		There exists a finite constant $B$ and a radius $r_*$ only depending on $\kcp,\gamma,d$
		such that for all $t>0$ and all $r\in(0,r_*)$
			\begin{align}\label{eq:upb}
						g(t,r) \le g_c(r) +Br^{2-d},
			\end{align}
		where $g_c(r)=(\Phi')^{-1}(-\frac{1}{2}r^2)$.
	\end{proposition}
Note that, in contrast to the lower bound in Proposition~\ref{prop:profile.lb}, the upper bound~\eqref{eq:upb} is formulated only for the limiting function $g$ obtained after sending $\ve\downarrow0$.
	\begin{proof}
			We adopt the notations of Proposition~\ref{prop:profile.lb} and its proof, where here it will suffice to consider the choice $\alpha=\frac{2}{\gamma}$. Thus, we let $\tilde g(r)=c_\gamma r^{-\frac{2}{\gamma}}$ and set
		\begin{align*}
			r_\ve= r_\ve(t)=\sup\{r>0: g_\ve(t,\rho)<\tilde g(\rho)\text{ for all }\rho\in(0,r)\}.
		\end{align*}
	Let $r_*$ be the radius obtained in Proposition~\ref{prop:profile.lb}.
For $r\in(0,r_\ve(t))$ we trivially have $g_\ve(t,r) \le \tilde g(r) = \big(\frac{\gamma}{2}r^2\big)^{-\frac{1}{\gamma}}$, or equivalently (cf.~\eqref{eq:Phi'})
\begin{align}\label{eq:449}
	\Phi'(g_\ve(r)) \le \Phi'(\tilde g(r)) = -\frac{1}{2}r^2(1+O(r^2)), \qquad r\in(0,r_\ve(t)).
\end{align}
The main step in the proof of the upper bound~\eqref{eq:upb} is to establish a bound similar to~\eqref{eq:449}  on the interval $r\in[r_\ve(t),r_*)$ in the case where
 $r_\ve(t)< r_*$. Of course, due to the possible formation of a point mass at the origin, 
 such a bound can in general only be expected to hold true up to some error term that tends to zero as $\ve\downarrow0$. 

If $r_\ve(t)< r_*$, we note that as in the proof of Proposition~\ref{prop:profile.lb} we have the formula 
		\begin{align}
			\frac{\dd}{\dd r}\Phi'(g_\ve) =\frac{\partial_rg_\ve}{h(g_\ve)} 
			&= -r\frac{h_\ve(g_\ve)}{h(g_\ve)} + b_\ve r^{1-d}\frac{1}{h(g_\ve)}.
		\end{align}
		Hence, for all $r\in[r_\ve,r_*)$,
		\begin{align}
		\Phi'(g_\ve(r)) &=\Phi'(\tilde g(r_\ve)) 
			-\int_{(r_\ve,r)}\rho\frac{h_\ve(g_\ve)}{h(g_\ve)}\,\dd\rho
			+ \int_{(r_\ve,r)}b_\ve \rho^{1-d}\frac{1}{h(g_\ve)}\,\dd\rho,
		\end{align}
where we omitted the (fixed) time argument $t$.
		To proceed, we define the set 
		\begin{align}
			J_\ve=J_\ve(t)=\{\rho\in[r_\ve(t),r_*):g_\ve(t,\rho)\ge {\ve}^{-1}\}.
		\end{align}
		On $[r_\ve,r_*)\setminus J_\ve$ we have $\frac{h_\ve(g_\ve)}{h(g_\ve)}\equiv 1$, while on $J_\ve$ we only know that $0\le\frac{h_\ve(g_\ve)}{h(g_\ve)}\le 1$.
		Hence, we may estimate for $r\in[r_\ve,r_*)$, using also the bound $g_\ve(\rho)\gtrsim \rho^{-\frac{2}{\gamma}}$ for $\rho\in(r_\ve,r_*)$ from Proposition~\ref{prop:profile.lb},
			\begin{align}\label{eq:450}
				\begin{aligned}
			\Phi'(g_\ve(r)) &\le\Phi'\big(c_\gamma r_\ve^{-\frac{2}{\gamma}}\big) 
			-\int_{(r_\ve,r)\setminus J_\ve}\rho\,\dd\rho
			+CK_*r^{2-d+(\gamma+1)\frac{2}{\gamma}}
			\\&\le -\frac{1}{2}r_\ve^2 + O(r^4)
			-\frac{1}{2}(r^2-r_\ve^2)+r_*\mathcal{L}^1(J_\ve)
			+CK_*r^{4-d+\frac{2}{\gamma}}
			\\&\le  			-\frac{1}{2}r^2
			+C(K_*)r^{4-d+\frac{2}{\gamma}}
			+r_*\mathcal{L}^1(J_\ve).			
		\end{aligned}
		\end{align}
	Here, we further used~\eqref{eq:Phi'} in the second step and $d>\frac{2}{\gamma}$ in the third step.

Combining~\eqref{eq:450} with~\eqref{eq:449}, we deduce 
(independently of whether $r_\ve<r_*$ or $r_\ve\ge r_*$)
	\begin{align}
	\Phi'(g_\ve(r)) \le	-\frac{1}{2}r^2\big(1+O(r^{2-d+\frac{2}{\gamma}})\big)
	+r_*\mathcal{L}^1(J_\ve)\qquad\text{ for all }r\in (0,r_*),
\end{align}
where $O=O(\cdot)$ only depends on $K_*$ and $\gamma$.
Mass conservation, i.e.\ $\int_{\mathbb{R}^d} f_\ve(t)\equiv \int_{\mathbb{R}^d}\fin$, implies that $\lim_{\ve\to0}\mathcal{L}^1(J_\ve(t))=0$. 
	Thus, sending $\ve\to0$, we infer the bound
	$\Phi'(g(r)) \le -\frac{1}{2}r^2(1+O(r^{2-d+\frac{2}{\gamma}}))$ for all $r\in(0,r_*)$.
	Finally, we invoke~\eqref{eq:Phi'.inv} and arrive at
		\begin{align}
		g(r) &\le c_\gamma r^{-\frac{2}{\gamma}}(1+O(r^{2-d+\frac{2}{\gamma}}))(1+O(r^2))
		=g_c(r)+O(r^{2-d}),\quad r\in (0,r_*).
	\end{align}
	\end{proof}
	
 For anisotropic data, the approximate solutions $\{f_\ve\}$ are dominated by an isotropic  scheme $\{\hat f_\ve\}$ (cf.\ Section~\ref{sssec:ani.ex}). Hence, the density $f(t,v)$ of the regular part of the limiting measure in Proposition~\ref{prop:lim} inherits the upper bound obtained above for the isotropic case.
	\begin{corollary}[Upper bound on space profile: anisotropic case]
		\label{cor:upper.ani}
		In addition to~\ref{it:HPspc}, \ref{it:HP.init} suppose that $\tfrac{2}{\gamma}+2-d>0$.
		There exists a finite constant $\hat B$ and a radius $\hat r_*$ only depending on $\fin$ (non-explicitly) and on $\gamma,d$ such that for all $t>0$ and all $v$ with $|v|\in(0,\hat r_*)$
		\begin{align*}
			f(t,v) \le f_c(v) +\hat B|v|^{2-d}.
		\end{align*}
	In particular, the point mass at the origin $t\mapsto \mu_t(\{0\})=m-\int f(t,\cdot)$ is continuous (as a consequence of Lebesgue's dominated convergence theorem).
	\end{corollary}

\subsection{Instantaneous regularisation}\label{ssec:inst.reg}

For the nonlinear problem~\eqref{eq:fpn} the Lebesgue space $L^{p_c}(\mathbb{R}^d)$, $p_c:=\frac{\gamma d}{2}$ is critical (as regards high values of the density). Thus, for $p>p_c$ one would expect equation~\eqref{eq:fpn} to enjoy a smoothing property in $L^{p}$. The following result formalises these heuristics. 

\begin{proposition}[Smoothing out subcritical singularities]\label{prop:ins.subcr}
	Let $\{f_\ve\}_{\ve\in(0,\ve_0]}$ be a family of (suitably regular) non-negative mild solutions of the $\ve$-regularised problems~\eqref{eq:fpnreg}\footnote{The family $\{f_\ve\}$ does not have to take the same initial data.} with uniformly controlled mass $\|f_\ve(t)\|_{L^1}\le m$.
	Let $p>p_c:=\frac{\gamma d}{2}$, let $t_0\ge0$, and suppose the following conditions:
	\begin{enumerate}[label=\textup{(C\arabic*)}]
		\item\label{it:L} There exists $L<\infty$ such that $\|f_\ve(t_0,\cdot)\|_{L^p(\mathbb{R}^d)}\le L$
		for all $\ve\in(0,\ve_0]$.
		\item\label{it:L'} There exists $ t_1\in(t_0,t_0+1]$, a constant $L'<\infty$ and a radius $r_0\in(0,1]$ such that $f_\ve(t,v)\le L'|v|^{-\frac{2}{\gamma}}$		
		for all $t\in[t_0,t_1]$, all $v\in\mathbb{R}^d$ with $|v|\le r_0$ and all $\ve\in(0,\ve_0]$.
	    \item\label{it:L''} For all $\tilde r_0>0$ there exists $L''=L''(\tilde r_0)<\infty$ such that $f_\ve\le L''$ 
	    in $[t_0,\infty)\times\{v:|v|\ge \tilde r_0\}$ for all $\ve\in(0,\ve_0]$.
	\end{enumerate}	
	Then there exists $T=T(L,L', L''(r_0), p,d,\gamma,m)\in(0,1]$ such that 
for $\hat T:=\min\{T,t_1{-}t_0\}$ and for all $\tau\in(0,\hat T]$ 
\begin{align}\label{eq:210}
	\sup_{\ve\in(0,\ve_0]}\;\sup_{t\in [t_0+\tau,t_0+\hat T]}\|f_\ve(t,\cdot)\|_{L^\infty(\mathbb{R}^d)}<\infty.
\end{align}
\end{proposition}

\begin{proof}
	We proceed in two steps. In a first step, we derive smoothing estimates based on the mild formulation~\eqref{eq:115} satisfied by $f_\ve$, 
	where as in Section~\ref{sec:ex} the nonlinear term is to be rewritten analogously to~\eqref{eq:116}.
	
	\underline{Step~1}: \textit{localised smoothing estimate}. 
	
Fix some sufficiently small $\epsilon_1>0$ such that 
\begin{align}\label{eq:700}
	\frac{p_c}{p}\le 1-2\epsilon_1.
\end{align}
	Let $\tilde p,\tilde q\in[p,\infty]$ with $\tilde p\le\tilde q$. Then
	\begin{align}\label{eq:def.a}
		a:=\frac{d\gamma}{4}\frac{1}{\tilde q}+\frac{1}{2} 
		= \frac{1}{2}\bigg(\frac{p_c}{\tilde q}+1\bigg)\le 1-\epsilon_1.
	\end{align}
Defining $b:=\frac{d}{2}(\frac{1}{\tilde p}-\frac{1}{\tilde q})(\frac{\gamma}{2}+1)$
we further have 
\begin{align}\label{eq:a+b}
	a+b-\frac{d}{2}\left(\frac{1}{\tilde p}-\frac{1}{\tilde q}\right) 
	&= \frac{d}{2}\left(\frac{1}{\tilde p}-\frac{1}{\tilde q}\right)\frac{\gamma}{2}
	+\frac{d\gamma}{4}\frac{1}{\tilde q}+\frac{1}{2} 
	=\frac{1}{2}\left(\frac{p_c}{\tilde p}+1\right)\le 1-\epsilon_1.
\end{align}
We assert that if $\tilde p,\tilde q$ are sufficiently close in the sense that
\begin{align}\label{eq:b.ineq}
	b=\tfrac{d}{2}(\tfrac{1}{\tilde p}-\tfrac{1}{\tilde q})(\tfrac{\gamma}{2}+1)\le 1-\hat\epsilon_1
\end{align}
for some $\hat\epsilon_1>0$, there exists an (explicit) strictly increasing function $\kappa\in C([0,1])$ only depending on $\epsilon_1,\hat\epsilon_1,d$ and on $L'$ and $L'':= L''(r_0)$ with $\kappa(0)=0$, and a finite constant $C_1=C_1(d)$ such that 
for all $t\in[0,t_1{-}t_0]$
\begin{align}\label{eq:step1}
	\|\chi_{\{|v|\le r_0\}}f_\ve\stn\|_{Z_t}\le C_1\|f_\ve(t_0,\cdot)\|_{L^{\tilde p}(\mathbb{R}^d)} 
	+ \kappa(t)\|f_\ve\stn\|_{Z_t}\big(\|f_\ve\stn\|_{Z_t}^\frac{\gamma}{2}+1\big),\qquad
\end{align}
where $f_\ve\stn(\tau,\cdot):=f_\ve(t_0+\tau,\cdot)$ and
\begin{align}
	\|\tilde f\|_{Z_t}:=	\|\tilde f\|_{Z_t^{(\tilde p,\tilde q)}}:=\sup_{s\in[0,t]}\nu(s)^{\frac{d}{2}(\frac{1}{\tilde p}-\frac{1}{\tilde q})}\|\tilde f(s,\cdot)\|_{L^{\tilde q}(\mathbb{R}^d)}.
\end{align}

\textit{Proof of Step 1.} 
Let $\zeta\in C^\infty_c(\mathbb{R}^d)$ with $0\le \zeta\le 1$, $\zeta=1$ on $\{|v|\le r_0\}$, 
$\supp\zeta\subset B_{2r_0}(0)$. 
	
		By the mild solution property of $f_\ve$, we have (cf.\ Sec.~\ref{ssec:mild})
	\begin{equation}\label{eq:116.}
		\begin{aligned}
			f_\ve\stn(\tau,v) 
			&=\int_{\mathbb{R}^d}\mathcal{F}(\tau,v,w)f_\ve(t_0,w)\,\dd w
			\\&\qquad +\int_0^\tau \ee^{-(\tau-s)}\int_{\mathbb{R}^d}\nabla_v\mathcal{F}(\tau{-}s,v,w)\cdot w\,\vt_\ve(f_\ve\stn(s,w))\,\dd w\dd s.
		\end{aligned}
	\end{equation}
	
 Using the bound $|\vt_\ve(g)|\le |g|^{\gamma+1}$ (cf.\ def.~\eqref{eq:555}), the fact that 
 $|w|f^\frac{\gamma}{2}_\ve(s,w)\le C(L',\gamma)$\,\footnote{Any dependence on the fixed parameter $\gamma$ will henceforth not be explicitly indicated.} 
 for $|w|\le r_0$ and $|f_\ve(s,w)|\le L''$ for $|w|\ge r_0$
 for all  $s\in[t_0,t_1]$ (cf.~\ref{it:L'} and~\ref{it:L''}),
 we now estimate for $0\le s<\tau\le t\le t_1{-}t_0$
\begin{align}
 	\bigg|\int_{\mathbb{R}^d}&\nabla_v\mathcal{F}(\tau{-}s,v,w)\cdot\big(w\,\vt_\ve(f_\ve\stn(s,w))\big)\,\dd w\bigg|
 	\\&\le C(L')
 	\int_{\{|w|\le r_0\}}|\nabla_v\mathcal{F}(\tau{-}s,v,w)|(f\stn_\ve)^{\frac{\gamma}{2}+1}(s,w)\,\dd w
 	\\&\quad +C(L'')\,\ee^{\tau{-}s}\int_{\{|w|>r_0\}}|\nabla_v\mathcal{F}(\tau{-}s,v,w)|\big(|v|+|\ee^{-(\tau{-}s)}w-v|\big)\,f_\ve\stn(s,w)\,\dd w.
 \end{align}
The integrals on the RHS will be handled similarly as in the proof of~\cite[Prop.~A.1]{CLR_2009}. 
 For estimating the $L^{\tilde q}(\mathbb{R}^d)$-norm, Young's convolution inequality is employed.
 For the first term on the RHS we invoke inequality~\eqref{eq:FPsemigr} and estimate
	\begin{align}
	\bigg\|	\int_{\mathbb{R}^d}&|\nabla_v\mathcal{F}(\tau{-}s,v,w)|(f\stn_\ve)^{\frac{\gamma}{2}+1}(s,w)\,\dd w\bigg\|_{L^{\tilde q}(\mathbb{R}^d)}
	\\&\le   C\nu(\tau{-}s)^{-\frac{1}{2}-\frac{d}{2}\left(\frac{\frac{\gamma}{2}+1}{\tilde q}-\frac{1}{\tilde q}\right)}
	\|(f\stn_\ve)^{\frac{\gamma}{2}+1}(s)\|_{L^{\tilde q/{(\frac{\gamma}{2}+1)}}(\mathbb{R}^d)}
	\\&\le  C\nu(\tau{-}s)^{-\frac{1}{2}-\frac{d\gamma}{4\tilde q}}
	\nu(s)^{-\frac{d}{2}(\frac{1}{\tilde p}-\frac{1}{\tilde q})(\frac{\gamma}{2}+1)}\|f\stn_\ve\|_{Z_t}^{\frac{\gamma}{2}+1}
	\\&\quad = C\nu(\tau{-}s)^{-a}\nu(s)^{-b}\|f\stn_\ve\|_{Z_t}^{\frac{\gamma}{2}+1}.
	\end{align}
Here and below, $C$ denotes a positive constant that only depends on fixed parameters, but which may change from line to line.

	We next estimate 
	\begin{align}
		\begin{aligned}
		\bigg\|		\zeta(v)\int_{\{|w|>r_0\}}&|\nabla_v\mathcal{F}(\tau{-}s,v,w)||v|\,f_\ve\stn(s,w)\,\dd w\bigg\|_{L^{\tilde q}(\mathbb{R}^d)}
			\\&\le C	\bigg\|	 \int_{\mathbb{R}^d}|\nabla_v\mathcal{F}(\tau{-}s,v,w)|f_\ve\stn(s,w)\,\dd w\bigg\|_{L^{\tilde q}(\mathbb{R}^d)}
			\\&\le C\nu(\tau{-}s)^{-\frac{1}{2}}\nu(s)^{-\frac{d}{2}(\frac{1}{\tilde p}-\frac{1}{\tilde q})}
			\|f\stn_\ve\|_{Z_t},
		\end{aligned}
	\end{align}
where the second step follows from~\eqref{eq:FPsemigr}.
	
	Finally, the rapid decay of the Fokker--Planck kernel allows us to further estimate
	 \begin{align}
		\begin{aligned}
			\bigg\|\int_{\mathbb{R}^d}|\nabla_v\mathcal{F}(\tau{-}s,v,w)|&|\ee^{-(\tau{-}s)}w-v|\,f_\ve\stn(s,w)\,\dd w\bigg\|_{L^{\tilde q}(\mathbb{R}^d)}
			\\&\le C\nu(\tau{-}s)^{-\frac{1}{2}}\nu(s)^{-\frac{d}{2}(\frac{1}{\tilde p}-\frac{1}{\tilde q})}
			\|f\stn_\ve\|_{Z_t},
		\end{aligned}
	\end{align}
see Lemma~\ref{l:rdfp} for details.	
	
	Inserting the above estimates into~\eqref{eq:116.}, we infer for $\tau\le t\in[0,t_1{-}t_0]\subseteq[0,1]$ 
	\begin{equation}\label{eq:200}
		\begin{aligned}
		&\nu(\tau)^{\frac{d}{2}(\frac{1}{\tilde p}-\frac{1}{\tilde q})}\|f_\ve\stn(\tau)\zeta\|_{L^{\tilde q}(\mathbb{R}^d)}\le
			C_1\|f_\ve(t_0)\|_{L^{\tilde p}(\mathbb{R}^d)}
			\\&\qquad +C(L',L'')\,\nu(\tau)^{\frac{d}{2}(\frac{1}{\tilde p}-\frac{1}{\tilde q})}\int_0^\tau \nu(\tau{-}s)^{-a}\nu(s)^{-b}\dd s\,
			\|f\stn_\ve\|_{Z_t}\big(\|f\stn_\ve\|_{Z_t}^\frac{\gamma}{2}+1\big),
		\end{aligned}
	\end{equation}
where we used once more inequality~\eqref{eq:FPsemigr} as well as the fact that $a\ge\frac{1}{2}$ and $b\ge\frac{d}{2}(\frac{1}{\tilde p}-\frac{1}{\tilde q})$.
To proceed we estimate for $t\in[0,1]$, using the bound $s\le2s\le\nu(s)\le 2\ee^{2}s$ for all $s\in[0,1]$
and a change of variables,
\small
\begin{align}
	\sup_{\tau\in[0,t]} C\,\nu(\tau)^{\frac{d}{2}(\frac{1}{\tilde p}-\frac{1}{\tilde q})}\int_0^\tau \nu(\tau{-}s)^{-a}\nu(s)^{-b}	\,\dd s
	&\le C\!\sup_{\tau\in[0,t]}\tau^{\frac{d}{2}(\frac{1}{\tilde p}-\frac{1}{\tilde q})+1-a-b}
		\int_0^1 (1{-}\tilde s)^{-(1-\epsilon_1)}\tilde s^{-(1-\hat\epsilon_1)}	\dd \tilde s
		\\&\le Ct^{\epsilon_1}
		\int_0^1 (1{-}\tilde s)^{-(1-\epsilon_1)}\tilde s^{-(1-\hat\epsilon_1)}	\dd \tilde s
	\;	\;=:\;\;\kappa(t),
\end{align}	
\normalsize
where we abbreviated $C=C(L',L'')\in(0,\infty)$.
In the first line we used inequality~\eqref{eq:def.a} and hypothesis~\eqref{eq:b.ineq},  in the second line we used~\eqref{eq:a+b}.

	Hence, for all $t\in[0,t_1{-}t_0]$
\begin{equation}\label{eq:201}
	\begin{aligned}
		\|\zeta\, f_\ve\stn\|_{Z_t}&\le C_1
		\|f_\ve(t_0)\|_{L^{\tilde p}(\mathbb{R}^d)}
		+\kappa(t)\|f\stn_\ve\|_{Z_t}\big(	\|f\stn_\ve\|_{Z_t}^\frac{\gamma}{2}+1\big),
	\end{aligned}
\end{equation}
	which proves the assertion of Step~1.

\bigskip

\underline{Step 2.} We are now ready to \textit{complete the proof of Proposition~\ref{prop:ins.subcr}} using estimate~\eqref{eq:step1} 
and property~\ref{it:L''} with $L'':=L''(r_0)$.
The idea is to perform a finite number of iterations in the integrability exponents to eventually upgrade the $\ve$-uniform $L^p$ bound on $f_\ve(t_0)$ to an $\ve$-uniform $L^\infty$ bound on $f_\ve(t_0+\tau)$ for given $\tau>0$ small. 

It is elementary to verify that for any $\tilde p\ge p>p_c$ and for $\tilde q:=2\tilde p$  
the tuple $(\tilde p,\tilde q)$ satisfies the hypotheses of Step~1 with parameter $\epsilon_1$ only depending on $p,\gamma,d$ and with $\hat\epsilon_1=\frac{1}{4}$. 
Indeed, for this choice we have
$\tfrac{d}{2}(\tfrac{1}{\tilde p}-\tfrac{1}{\tilde q})(\tfrac{\gamma}{2}+1)
=\tfrac{p_c}{\tilde p}\tfrac{1}{2}(\tfrac{1}{2}+\tfrac{1}{\gamma})
\le \frac{3}{4}=1-\hat \epsilon_1$, showing~\eqref{eq:b.ineq}.
Hence, by Step~1, there exists a strictly increasing function $\kappa\in C([0,1])$ with $\kappa(0)=0$ only depending on $p,\gamma,d$ and on $L',L''$
such that for any $\tilde p\ge p$, for $\tilde q=2\tilde p$, and all $t\in[0,t_1{-}t_0]$
\begin{align}\label{eq:step1.appl}
	\|\chi_{\{|v|\le r_0\}}f_\ve\stn\|_{Z_t}\le C_1\|f_\ve(t_0,\cdot)\|_{L^{\tilde p}(\mathbb{R}^d)} 
	+ \kappa(t)\|f_\ve\stn\|_{Z_t}\big(\|f_\ve\stn\|_{Z_t}^\frac{\gamma}{2}+1\big),
\end{align}
where $f_\ve\stn(\tau,\cdot):=f_\ve(t_0+\tau,\cdot)$ and $\|\cdot\|_{Z_t}:=	\|\cdot\|_{Z_t^{(\tilde p, \tilde q)}}$.

In the following, we abbreviate 
$F(t):=\|\chi_{\{|v|\le r_0\}}f_\ve\stn\|_{Z_t^{(\tilde p, \tilde q)}}$. 
Thanks to mass control and  \ref{it:L''}, we have for any $t\in[0,1]$ and any $\tilde q\ge\tilde p\ge 1$ the estimate
\begin{align}\label{eq:fZ}
	\begin{aligned}
	\|f\stn_\ve\|_{Z_t}&\le	F(t) 
	+ \sup_{s\in[0,t]}\nu(s)^{\frac{d}{2}(\frac{1}{\tilde p}-\frac{1}{\tilde q})}\|\chi_{\{|v|>r_0\}}f_\ve^{(t_0)}(s)\|_{L^{\tilde q}(\mathbb{R}^d)}
	\\& \le	F(t) +
	\sup_{s\in[0,t]}\nu(s)^{\frac{d}{2}(\frac{1}{\tilde p}-\frac{1}{\tilde q})}(m+L'')
	\\& \le	F(t) +\CC,
	\end{aligned}
\end{align}
where $\CC=C(d)(m+L'')$.

Inequalities~\eqref{eq:step1.appl} and~\eqref{eq:fZ} show that for $(\tilde p,\tilde q)=(p,2p)$ the function $F(t)$ obeys a bound of the form
\begin{align}\label{eq:784}
	F(t)\le B
	+\kappa(t)\big(F(t)+\CC\big)\big((F(t)+\CC)^{\frac{\gamma}{2}}+1\big), \qquad t\in[0,t_1{-}t_0],
\end{align}
where $B$ only depends on fixed parameters (here one may choose $B=C_1L$).
Since $\kappa(0)=0$, there exists, for every $B>0$, a unique maximal time $T_B\in(0,1]$ 
such that 
\begin{align}
	\sup_{t\in[0,T_B]}\kappa(t)\le \frac{B}{(2B+\CC)\big((2B+\CC)^\frac{\gamma}{2}+1\big)},
\end{align}
that is $T_B:=\kappa^{-1}(\min\{\kappa(1),\frac{B}{(2B+\CC)\big((2B+\CC)^{\gamma/2}+1\big)}\})$, where $\kappa^{-1}$ denotes the inverse of~$\kappa$. 
With this choice, we deduce from~\eqref{eq:784} that $F(t)\le 2B$ for all $t\in[0,\hat T_{B}]$, where $\hat T_{B}:=\min\{T_B,t_1{-}t_0\}$. In particular, for $B=B_1:=C_1L$ we infer
\[
\|\chi_{\{|v|\le r_0\}}f_\ve(t_0+t)\|_{L^{2p}}
\le\nu(t)^{-{\frac{d}{4p}}}2B_1,\qquad t\in[0,\hat T],
\]
where $\hat T:=\hat T_{B_1}$.
Combined with~\ref{it:L''} and mass control this shows that 
\begin{align}\label{eq:702}
	\sup_{\ve\in(0,\ve_0]}\|f_\ve^{(t_0)}\|_{L^\infty([\tau,\hat T];L^{2p}(\mathbb{R}^d))}
	\le C(\tau)\quad \text{ for all }\tau\in (0,\hat T]
\end{align}
for some non-increasing function $C(\cdot)$, which depends on further fixed parameters.
This argument can be iterated to give the asserted bound~\eqref{eq:210} 
for the same time $\hat T\ (=\hat T_{B_1})$.
Let us provide some details.
Fix some $N\in \mathbb{N}_+$ large enough such that 
$2^Np>d$.
For $\tau>0$ small and $(\tilde p,\tilde q)=(2^ip,2^{i+1}p)$, $i=1$, the (time-shifted) function 
$F(t):=\|\chi_{\{|v|\le r_0\}}f_\ve^{(t_0+\tau)}\|_{Z_t}$ obeys a bound of the form 
\begin{align}\label{eq:701}
	F(t)\le B
	+\kappa(t)\big(F(t)+\CC\big)\big((F(t)+\CC)^{\frac{\gamma}{2}}+1\big), \qquad t\in[0,t_1{-}(t_0{+}\tau)]
\end{align}
where $B=B(\tau)<\infty$ is non-increasing in $\tau>0$.
This allows us to infer (for $i=1$) that
\begin{align}\label{eq:703}
	\sup_{\ve\in(0,\ve_0]}\|f_\ve^{(t_0)}\|_{L^\infty([\tau,\hat T];L^{2^{i+1}p}(\mathbb{R}^d))}
	\le C(\tau)\quad \text{ for all }\tau\in (0,\hat T]
\end{align}
for a non-increasing function $C(\cdot)$. Observe that, thanks to the non-increase with respect to $\tau$ of the constants $C(\cdot)$ and $B$ appearing in~\eqref{eq:702} and~\eqref{eq:701}, the locally uniform bound~\eqref{eq:703} can indeed be achieved on the entire time interval $(0,\hat T]$ (by iteration), so that the final time $\hat T$ does not need to be decreased.
Repeating the argument for $i=2,\dots,N{-}1$, we deduce a bound of the form
$\sup_{\ve\in(0,\ve_0]}\|f_\ve^{(t_0)}\|_{L^\infty([\tau,\hat T];L^{2^Np}(\mathbb{R}^d))}
\le C(\tau)$ for all $\tau\in (0,\hat T].$
For $\tau\in(0,\hat T)$ and $f\stn_\ve$ replaced by $f_\ve^{(t_0+\tau)}$ we may now take $\tilde p:=2^Np>\max\{d,p_c\}$ in Step~1, in which case 
the choice $\tilde q=\infty$ is admissible. (Indeed, with this choice we have 
$b=\frac{d}{2}(\frac{\gamma}{2}\frac{1}{\tilde p}+\frac{1}{\tilde p})< \frac{d}{2}(\frac{1}{d}+\frac{1}{d})=1$, so that~\eqref{eq:b.ineq} is fulfilled.) 
Arguing similarly as before we infer~\eqref{eq:210}.
\end{proof}

\subsection{Space profile}\label{ssec:profile.proof}

Finally, we are in a position to prove Theorem~\ref{thm:profile.r}.

\begin{proof}[Proof of Theorem~\ref{thm:profile.r}]
	Thanks to the short-time regularity for~\eqref{eq:fpn}, we may assume without loss of generality the strengthened version~\ref{eq:hpinit} of~\ref{it:HP.init}.
We fix some $\ul\alpha<\alpha_c:=\frac{2}{\gamma}$ as in Proposition~\ref{prop:profile.lb} and let $r_*>0$ denote the associated radius obtained in Proposition~\ref{prop:profile.lb}. 
Let now $\hat t>0$.  We assert that the behaviour of $g(\hat t,\cdot)$ near zero is determined by the fact of whether or not the hypotheses of \textit{Case~1} are fulfilled, where Case~1 is determined as follows:

\medskip

 \underline{Case 1:} there exists $\alpha\in[\ul\alpha,\alpha_c)$, a time $\hat t_0<\hat t $, a radius $r_0\in(0,r_*)$,  and $\ve_0\in(0,\epsilon_*]$ such that for all $\ve\in(0,\ve_0]$, all $r\in(0,r_0]$ and all $t\in[\hat t_0,\hat t ]$
		\begin{align}\label{eq:511}
			g_\ve(t,r)\le \tilde g^{(\alpha)}(r):=c_\gamma r^{-\alpha},\;\text{ where } c_\gamma:=\big(\tfrac{2}{\gamma}\big)^\frac{1}{\gamma}.
		\end{align}
	Here, $\{g_\ve\}$ denotes the family of isotropic approximate solutions in radial coordinates.
	
If Case~1 is fulfilled, Proposition~\ref{prop:ins.subcr} implies the existence of a constant $\delta>0$ such that
\begin{align}\label{eq:785}
	\sup_{\ve\in(0,\ve_0]}\sup_{t\in (\hat t-\delta,\hat t]}	\|f_\ve(t,\cdot)\|_{L^\infty}<\infty.
\end{align}
Indeed, since $\alpha<\alpha_c$, we can choose $p>p_c$ such that 
$f^{(\alpha)}(v):=c_\gamma|v|^{-\alpha}\in L^p(B_1)$, where $B_1:=\{v:|v|\le 1\}$.
Hence, combining~\eqref{eq:511} with mass conservation and the uniform bound away from the origin (cf.\ Lemma~\ref{l:boundiso}), 
we find that
\begin{align}
	\sup_{\ve\in(0,\ve_0]}	\|f_\ve(t,\cdot)\|_{L^p}\le L
\end{align}
for all $t\in[\hat t_0,\hat t]$ and some finite constant $L$. 
Property~\eqref{eq:511} further guarantees the bound 
$\sup_{\ve\in(0,\ve_0]} f_\ve(t,v)\le L'|v|^{-\frac{2}{\gamma}}$ for all $v\in B_{r_0}$, all $t\in[\hat t_0,\hat t ]$ and suitable $L'<\infty$. Finally, Lemma~\ref{l:boundiso} ensures that for all $\tilde r_0>0$ there exists $L''(\tilde r_0)<\infty$ such that $\sup_{t>0}\sup_{\ve\in(0,\ve_0]} f_\ve(t,v)\le L''(\tilde r_0)$ whenever $|v|\ge \tilde r_0$.
Hence, for every $t_0\in[\hat t_0,\hat t)\cap[\hat t{-}1,\hat t)$ and for $p$ as above the conditions~\ref{it:L}--\ref{it:L''} of Proposition~\ref{prop:ins.subcr} are satisfied with $t_1=\hat t$, which implies~\eqref{eq:785} for suitable $\delta>0$.

It is easy to see that, after possibly decreasing $\delta>0$,  the bound~\eqref{eq:785} even holds with $\sup_t$ being taken over $t\in J_{\hat t}:=(\hat t-\delta,\hat t+\delta)$. (To this end, one may adapt the estimates in the proof of Proposition~\ref{prop:ins.subcr} and choose $p=\infty$ in~\ref{it:L}. In this case, the proof greatly simplifies,
condition~\ref{it:L'} is not needed and one may deduce an estimate of the form~\eqref{eq:210} even with $\tau=0$.)
Given this uniform bound, we can argue classically as in the proof of Proposition~\ref{prop:lim}~\ref{it:approxProp} to infer the smoothness of the limiting density $f$ on $J_{\hat t}\times\mathbb{R}^d\supset J_{\hat t}\times B_1$.

\medskip

\underline{Case 2:} it remains to consider the situation where the hypotheses of Case~1 are not satisfied.
In this case, we can find sequences 
\begin{align}
	\alpha_j\uparrow \alpha_c,\qquad t_j\uparrow \hat t ,\qquad r_j\downarrow 0,\qquad  \ve_j\downarrow 0,
\end{align}
with  $\ul\alpha\le\alpha_j$ and $r_j<r_*$ for all $j$, in such a way that 
\begin{align}
	g_{\ve_j}(t_j,r_j)\ge \tilde g^{(\alpha_j)}(r_j)\qquad\text{ for all }j\in \mathbb{N}.
\end{align}
Thus, invoking Proposition~\ref{prop:profile.lb}, we infer
\begin{align}\label{eq:122}
	g_{\ve_j}(t_j,r)\ge \tilde g^{(\alpha_j)}(r)-Cr^{2-d}\qquad\text{ for all }r\in(r_j,r_*).
\end{align}
By construction,  $\lim_{j\to\infty}g_{\ve_j}(t_j,r)=g(\hat t ,r)$ for every $r>0$.
Hence, sending $j\to\infty$ in inequality~\eqref{eq:122} yields
\begin{align}
	g(\hat t ,r)\ge  \tilde g^{(\alpha_c)}(r)-Cr^{2-d}\qquad\text{ for all }r\in(0,r_*),
\end{align}
which implies that
\begin{align}\label{eq:123}
	g(\hat t ,r)\ge  g_c(r)-Cr^{2-d}\qquad\text{ for all }r\in(0,r_*).
\end{align}
In view of the upper bound in Proposition~\ref{prop:upperBd} this completes the proof of the main assertion in Theorem~\ref{thm:profile.r}.

Let now $\alpha=\alpha_c$ in Proposition~\ref{prop:profile.lb} and define $\tilde r_\ve$ correspondingly. If $\hat t$ is such that $\mu_{\hat t}(\{0\})>0$, we must have $\lim_{\ve\to0}\tilde r_\ve(\hat t)=0$.
Proposition~\ref{prop:profile.lb} (combined with Prop.~\ref{prop:upperBd}) thus implies the assertion concerning this case.
\end{proof}

\section{Renormalised form}\label{sec:renorm}

\subsection{Variational structure}

Our subsequent analysis  relies on the following gradient-flow structure of
the regularised Fokker--Planck equation~\eqref{eq:fpnreg}. Such a structure has previously been used in~\cite[Section~3.3]{CHR_2020} for the proof  of an energy dissipation identity.
To proceed, let us recall that $\Phi'(s)=-\int_s^\infty\frac{1}{h(\sigma)}\,\dd\sigma$ and $\Phi(0)=0$.

We define the approximate free energy functional by
\begin{align}
	\Entr_\varepsilon(f) =\int_{\mathbb{R}^d}\left(\frac{|v|^2}{2}f+ \Phi_\varepsilon(f)\right)\mathrm{d}v,
\end{align}
where $\Phi_\varepsilon\in C([0,\infty))\cap C^\infty((0,\infty))$ satisfies
\begin{align}\label{eq:588}
	\Phi_\varepsilon(s) = \Phi(s)\text{ for }s\in[0,\ve^{-1}]
\end{align}
and
\begin{align}\label{eq:589}
	\Phi_\varepsilon''=\frac{1}{h_\ve}, \quad \Phi_\ve\ge\Phi.
\end{align}
The function $\Phi_\varepsilon$ with the above properties can be obtained by setting
	$\Phi_\ve(s)=\int_0^s\Phi_\ve'(\sigma)\,\dd\sigma,$
where $\Phi_\ve'(s)$ is given by 
\begin{align*}
	\Phi_\ve'(s) = -\int_s^{B_\ve}\frac{1}{h_\ve(\sigma)}\,\dd\sigma
\end{align*}
with the constant $B_\ve>\frac{1}{\ve}$ being such that
\begin{align*}
	\int_\frac{1}{\ve}^{B_\ve}\frac{1}{h_\ve(\sigma)}\,\dd\sigma = \int_\frac{1}{\ve}^\infty\frac{1}{h(\sigma)}\,\dd\sigma.
\end{align*}
Identity~\eqref{eq:588} is a consequence of the fact that $h_\ve(s)=h(s)$  in $[0,\ve^{-1}]$, while the second property in~\eqref{eq:589} follows from the inequality $h_\ve\le h$.

 Notice that the functional derivative of $\mathcal{H}_\ve$ is given by 
$\delta\mathcal{ H}_\ve(f)=\frac{1}{2}|v|^2+\Phi_\ve'(f)$, 
which allows us to rewrite~\eqref{eq:fpnreg} as
\begin{align}
	\partial_tf_\ve = \divv\big(h_\ve(f_\ve)\nabla \delta \Entr_\ve(f_\ve)\big).
\end{align}

\begin{lemma}[Energy dissipation balance for~\texorpdfstring{\eqref{eq:fpnreg}}{}]\label{l:edi.reg}
Under the hypotheses of Proposition~\ref{prop:ex.reg}, 
	the solutions $f_\ve$ of~\eqref{eq:fpnreg} obtained therein satisfy for all $0\le s\le t <\infty$
	\begin{align}\label{eq:290}
		\Entr_\ve(f_\ve(t)) +\int_s^t\int_{\mathbb{R}^d}\frac{1}{h_\ve(f_\ve)}|\nabla f_\ve+vh_\ve(f_\ve)|^2\dd v\,\dd\tau = \Entr_\ve(f_\ve(s)).
	\end{align}
\end{lemma}

\begin{proof} 
Recall that $f_\ve\in C^{1,2}((0,\infty)\times\mathbb{R}^d)$ is a classical solution of~\eqref{eq:fpnreg}.
Hence, the only task in deriving equation~\eqref{eq:290} lies in appropriately controlling the tails as $|v|\to\infty$. 
This is a consequence of the moment control of the bounded function $f_\ve$ and follows from classical arguments, see e.g.~\cite{CLR_2009}.  
\end{proof}

\begin{lemma}\label{l:edinH}
	Suppose~\ref{it:HPspc},~\ref{it:HP.init} and use the notations in Proposition~\ref{prop:lim}.
	For any $t\ge0$
	\begin{align*}
		\liminf_{\ve\to0}	\Entr_\ve(f_\ve(t)) \ge  	\Entr(f(t)),
	\end{align*}
where the $\liminf$ is taken along the sequence $\ve\downarrow0$ selected in Proposition~\ref{prop:lim}~\ref{it:approxProp}.
\end{lemma}
\begin{proof}
	Since $\Phi_\ve\ge\Phi$ (cf.~\eqref{eq:589}), we have
	$\int\Phi_\ve(f_\ve(t))\ge 	\int\Phi(f_\ve(t))$.
	Next, given $\delta>0$, we let $L=L(\delta)>0$ be large enough such that $|\Phi(s)|\le\delta s$ for $s\ge L$. Then
	\begin{align*}
		\int\Phi(f_\ve(t,v))\,\dd v &\ge 	\int\Phi(f_\ve(t,v))\chi_{\{f_\ve\le L\}}\,\dd v +	\int\Phi(f_\ve(t,v))\chi_{\{f_\ve> L\}}\,\dd v 
		\\&\ge 	\int\Phi(f_\ve(t,v))\chi_{\{f_\ve\le L\}}\,\dd v -\delta m,
	\end{align*}
where $m=\int\fin$.
	Sending first $\ve\to0$ (using dominated convergence) and then $L\to\infty$, we infer 
		$\liminf_{\ve\to0}\int\Phi_\ve(f_\ve(t)) \ge 	\int\Phi(f(t))-\delta m$ and hence
	\begin{align*}
		\liminf_{\ve\to0}\int\Phi_\ve(f_\ve(t))&\ge 	\int\Phi(f(t)).
	\end{align*}

For the kinetic part, we let $\mathcal{A}_{\rho}:=\{\rho\le |v|\le \rho^{-1}\}$ for $0<\rho\ll1$
and estimate
\begin{align}
\left|\int_{\mathbb{R}^d}\tfrac{1}{2}|v|^2f_\ve(t)
-\int_{\mathbb{R}^d}\tfrac{1}{2}|v|^2f(t)\right|
&\le 
	\left|\int_{\mathcal{A}_{\rho}} \tfrac{1}{2}|v|^2(f_\ve(t)-f(t))\,\dd v \right|
	+\rho^2m+ C\rho\|\fin\|_{L^1_3},
\end{align}
where we used mass conservation and the  bound $\|f(t)\|_{L^1_3},$ $\|f_\ve(t)\|_{L^1_3}\le C\|\fin\|_{L^1_3}<\infty$ (cf.\ L.~\ref{l:moments}).
We deduce
$\lim_{\ve\to0}\int_\dom \frac{1}{2}|v|^2f_\ve(t,v)\,\dd v = \int_\dom\frac{1}{2} |v|^2f(t,v)\,\dd v$, where we used
the locally uniform convergence in Prop.~\ref{prop:lim}~\ref{it:approxProp}
 and the fact that $0<\rho\ll1$ can be taken arbitrarily small.
\end{proof}

\subsection{The limiting measure is a renormalised solution}\label{ssec:renorm}

\begin{proof}[Proof of Theorem~\ref{thm:renorm}] 
The weak-star continuity of the mass-conserving curve $t\mapsto\mu_t$ in $\mathcal{M}_+(\mathbb{R}^d)$ has already been established in Proposition~\ref{prop:lim}.
	
	We next show that $\mathcal{T}_k(f)=\min\{f,k\}$ has a weak derivative $\nabla\mathcal{T}_k(f)\in L^2_\loc([0,\infty)\times\mathbb{R}^d)$.
For this purpose, we choose $s=0$ and $t=T$ in estimate~\eqref{eq:290} and, letting $\epsilon_*>0$ be small enough so that, by Lemma~\ref{l:edinH}, $-\Entr_\ve(f_\ve(T))\le -\Entr(f(T))+1$ for all $\ve\in(0,\epsilon_*]$, we infer the $\ve$-uniform bound
\begin{align}\label{eq:289}
	\int_0^T\!\!\int_{\mathbb{R}^d}\frac{1}{h_\ve(f_\ve)}|\nabla f_\ve+vh_\ve(f_\ve)|^2\dd v\,\dd\tau \le \Entr(\fin)-\Entr(f(T))+1.
\end{align}
To deduce a bound on $\nabla\mathcal{T}_k(f_\ve)$, we note that
\begin{align}
	|\nabla \mathcal{T}_k(f_\ve)|^2
	\le 2|\mathcal{T}_k'(f_\ve)[\nabla f_\ve+vh_\ve(f_\ve)]|^2
	+ 2|\mathcal{T}_k'(f_\ve)vh_\ve(f_\ve)|^2.
\end{align}
Hence, using the fact that $|\mathcal{T}_k'|\le1$ and $\mathcal{T}_k'(s)=0$ for $s>k$, we deduce from~\eqref{eq:289} for any $R\in(0,\infty)$
\begin{equation}\label{eq:585}
	\int_0^T\!\!\int_{\{|v|\le R\}}|\nabla \mathcal{T}_k(f_\ve)|^2\,\dd v\dd t
	\le C(k)(\Entr(\fin)-\Entr(f(T))+1)+C(k,R)T.
\end{equation}
Thanks to the convergence in Proposition~\ref{prop:lim}~\ref{it:approxProp}
\begin{align}
	&	\mathcal{T}_k(f_\ve)\to \mathcal{T}_k(f) \quad  \text{ a.e.\ in }[0,\infty)\times\mathbb{R}^d,
	\\&	\mathcal{T}_k(f_\ve)\to \mathcal{T}_k(f) \quad  \text{ in }L^p_\loc([0,\infty)\times\mathbb{R}^d), \quad\text{for all }p\in[1,\infty),
	\\\text{and  thus, by }\eqref{eq:585}, \qquad&\nabla \mathcal{T}_k(f_\ve)\weak\nabla \mathcal{T}_k(f)\quad\text{ in }L^2_\loc([0,\infty)\times \mathbb{R}^d).\label{eq:584}
\end{align}
As a consequence,
\begin{equation}\label{eq:585'}
	\int_0^T\!\!\int_{\{|v|\le R\}}|\nabla \mathcal{T}_k(f)|^2\,\dd v\dd t
	\le C(k)(\Entr(\fin)-\Entr(f(T))+1)+C(k,R)T.
\end{equation}
If the gradients $\nabla \mathcal{T}_k(f_\ve)$ were known to converge strongly in $L^2_\loc$, the renormalised formulation~\eqref{eq:renorm} could easily be derived from that for $f_\ve$ in the limit $\ve\to0$. For general anisotropic solutions such a result is, however, not available at the moment. The proof of~\eqref{eq:renorm} presented below in the isotropic case uses a somewhat different argument that will be taken up when deriving the entropy balance law~\eqref{eq:edb}.

Let now $\xi\in C^\infty([0,\infty))$ have a compactly supported derivative $\xi'$,
let $T<\infty$ and let $\psi\in C^\infty_c([0,T]\times\dom)$. 
Further let $\vp\in C^\infty([0,\infty);[0,1])$ satisfy $\vp(r)=0$ for $r\in[0,1]$ and $\vp(r)=1$ for $r\ge 2$, and abbreviate $\vp_\rho(r)=\vp(r/\rho)$ for $\rho\in(0,1]$. Then, since $f$ is a classical solution of~\eqref{eq:fpn} in $(0,\infty)\times\big(\mathbb{R}^d\setminus\{0\}\big)$,
a direct calculation gives
\begin{align}\label{eq:298.}
	\begin{aligned}
		\int_\dom \xi(f(T,\cdot))&\psi(T,\cdot)\vp_\rho(|v|)\,\dd v-\int_\dom \xi(\fin)\psi(0,\cdot)\vp_\rho(|v|)\,\dd v
		-\int_0^T\!\!\int_\dom \xi(f)\partial_t\psi\vp_\rho(|v|)\,\dd v\dd t
		\\&=-\int_0^T\!\!\int_\dom (\nabla f+h(f)v)\cdot [\xi''(f)\nabla f\psi\vp_\rho(|v|)+\xi'(f)\nabla\psi \vp_\rho(|v|)]\,\dd v\dd t
		\\&\quad -\int_0^T\!\!\int_\dom (\nabla f+h(f)v)\cdot [\xi'(f)\psi \vp_\rho'(|v|)\cdot\tfrac{v}{|v|} ]\,\dd v\dd t.
	\end{aligned}
\end{align}
By the dominated convergence theorem and since $\vp_\rho(r)\overset{\rho\downarrow0}{\to}1$ for all $r>0$, the LHS of~\eqref{eq:298.} converges, as $\rho\to0$, to 
\begin{align}\label{eq:298.1}
	\begin{aligned}
	\int_\dom \xi(f(T,\cdot))\psi(T,\cdot)\,\dd v-\int_\dom \xi(\fin)\psi(0,\cdot)\,\dd v
		-\int_0^T\!\!\int_\dom \xi(f)\partial_t\psi\,\dd v\dd t.
	\end{aligned}
\end{align}
Likewise, thanks to the bound~\eqref{eq:585'} and the compact support of $\xi'',\xi'$ and of $\psi$, the dominated convergence theorem allows to pass to the limit in the first integral on the RHS of~\eqref{eq:298.} giving the term
\begin{align}\label{eq:298.2}
	\begin{aligned}
	-\int_0^T\!\!\int_\dom (\nabla f+h(f)v)\cdot [\xi''(f)\nabla f\psi+\xi'(f)\nabla\psi ]\,\dd v\dd t.
	\end{aligned}
\end{align}
We are left to show that the last integral in~\eqref{eq:298.} vanishes in the limit $\rho\downarrow0$.
First, since $|h(f)\xi'(f)|\le C(\supp\xi')<\infty$ and $\vp_\rho'(|v|)=0$ for $|v|\ge2\rho$ as well as $|v\vp_\rho'(|v|)|=|\rho^{-1}v\vp'(|\rho^{-1}v|)|\lesssim1$, 
the dominated convergence theorem  yields
\begin{align}
\Big|\int_0^T\!\!\int_\dom h(f)v\cdot &[\xi'(f)\psi \vp_\rho'(|v|)\cdot\tfrac{v}{|v|} ]\,\dd v\dd t\Big|
\\	&\le \int_0^T\!\!\int_\dom |h(f)\xi'(f)||\psi||v\vp_\rho'(|v|)|\,\dd v\dd t \;\longrightarrow\; 0
\text{ as }\rho\to0.
\end{align}
The remaining part of the integral is more delicate. We estimate using the radial symmetry of $f(t,v)\;(=:g(t,|v|))$
\begin{align}
	\Big|\int_0^T\!\!\int_\dom& \nabla f\cdot [\xi'(f)\psi \vp_\rho'(|v|)\cdot\tfrac{v}{|v|} ]\,\dd v\dd t\Big|
	\\&\le C\int_0^T\!\!\int_0^{2\rho} |\xi'(g)\partial_rg| \rho^{-1}|\vp'(\rho^{-1}r)|r^{d-1}\dd r\dd t
\\&=C\int_0^TA(t,\rho)\,\dd t,
\end{align}
where we abbreviated 
\begin{align}
	 A(t,\rho)=\int_0^{2\rho} |\xi'(g)\partial_rg| \rho^{-1}|\vp'(\rho^{-1}r)|r^{d-1}\dd r.
\end{align}
\label{page:Alim.ren}
As a consequence of the bound~\eqref{eq:apriori}, we have $|\xi'(g)r^{d-1}\partial_rg|\le C\kcp+C(\supp\xi')r^d$. We hence infer the following $(t,\rho)$-uniform bound on $|A(t,\rho)|:$
\begin{align}
	|A(t,\rho)|\le C \int_0^{2\rho}\rho^{-1}|\vp'(\rho^{-1}r)|\,\dd r=C \int_0^{2}|\vp'(\hat r)|\,\dd \hat r.
\end{align}
Thus, to show that $\lim_{\rho\to0}\int_0^TA(t,\rho)\,\dd t=0$ it suffices to prove
 the pointwise convergence $\lim_{\rho\to0}A(t,\rho)=0$ for (almost) all $t\in(0,T]$.
  Thanks to Theorem~\ref{thm:profile.r}, only the following two cases may occur.
 
 \medskip

\underline{Case 1: $g(t,0+)=+\infty$.}
In this case, there exists $r_*>0$ such that $\xi'(g(t,r))=0$ for all $r\in(0,r_*)$. Hence, we trivially have $\lim_{\rho\to0} A(t,\rho)=0$.
\smallskip

\underline{Case 2: $g(t,\cdot)\in L^\infty$.}
In this case, by Theorem~\ref{thm:profile.r}, there exists a neighbourhood $J$ of $t$ such that $f_{|J\times\mathbb{R}^d}$ is smooth. In particular, $\partial_rg(t,\cdot)\in L^\infty(0,1)$. 
If $d>1$, the conclusion $\lim_{\rho\to0} A(t,\rho)=0$ then directly follows from the definition of $A(t,\rho)$, while for $d=1$ we resort to the fact that $\sup_{r\in(0,\rho)}|\partial_rg(t,r)|\to 0$ as $\rho\to0$.
\end{proof}

\subsection{Energy dissipation identity}\label{ssec:edi}

An argument similar to that in the proof of Theorem~\ref{thm:renorm} shows that isotropic solutions satisfy the energy dissipation balance. In the anisotropic case, we obtain  an inequality.
	
	\begin{proof}[Proof of Proposition~\ref{prop:edb}]
		Combining Lemmas~\ref{l:edi.reg} and~\ref{l:edinH} with the convergence properties of $f_\ve$ to $f$ in Proposition~\ref{prop:lim}~\ref{it:approxProp}, we readily infer for all $t>0$ the inequality
		\begin{align}
		\Entr(f(t))+	\int_0^t\mathcal{D}(f(\tau))\,\dd\tau \le \Entr(\fin),
		\end{align}
	where 
	\[
	\mathcal{D}(f):=\int_\dom \tfrac{1}{h(f)}|\nabla f+h(f)v|^2\,\dd v.
	\]	
		It remains to prove that in the isotropic case the above inequality holds with an equality.
		Then, the asserted identity~\eqref{eq:edb} follows by subtracting on both sides the quantity $\Entr(f(s))$, which is then known to equal $\Entr(\fin)-\int_0^s\mathcal{D}(f(\tau))\,\dd\tau$.
		Thus, in the remainder, we assume that $\fin$ is isotropic. 
		Moreover, without loss of generality, we may assume  hypotheses~\ref{eq:hpinit}. Otherwise we replace $\fin$ by $f(t_0)$ for small $t_0>0$. From the arguments below we will then obtain the identity
		$	\Entr(f(t))+	\int_{t_0}^t\mathcal{D}(f(\tau))\,\dd\tau = \Entr(f(t_0)),$ and taking the limit $t_0\downarrow0$, using monotone convergence for the term $	\int_{t_0}^t\mathcal{D}(f(\tau))\,\dd\tau$ and dominated convergence for $\Entr(f(t_0))$	we will arrive at the assertion.
		
As in the proof of Theorem~\ref{thm:renorm} (cf.\ Section~\ref{ssec:renorm}), we pick a non-decreasing function $\vp\in C^\infty([0,\infty);[0,1])$ satisfying $\vp(r)=0$ for $r\in[0,1]$ and $\vp(r)=1$ for $r\ge 2$, and abbreviate $\vp_\rho(r)=\vp(r/\rho)$ for $\rho\in(0,1]$. 
	Then, defining
		\begin{align}
		\Entr^{(\rho)}(f):=\int_{\dom}[\,\tfrac{1}{2}|v|^2f+\Phi(f)\,]\vp_\rho(|v|)\,\dd v,
	\end{align}
	one has
	\begin{align}
		\Entr^{(\rho)}(f(t))-\Entr^{(\rho)}(\fin)&
		=\int_0^t\!\int_{\dom}[\,\tfrac{1}{2}|v|^2+\Phi'(f)\,]\,\divv(\nabla f+h(f)v)\vp_\rho(|v|)\,\dd v\dd\tau
		\\&=-\int_0^t\!\int_{\dom} \frac{1}{h(f)}|\nabla f+h(f)v|^2\vp_\rho(|v|)\,\dd v\dd\tau
		\\&\qquad-\int_0^t\!\int_{\dom}[\,\tfrac{1}{2}|v|^2+\Phi'(f)\,]\,(\nabla f+h(f)v)\cdot\tfrac{v}{|v|}\vp_\rho'(|v|)\,\dd v\dd\tau.
	\end{align}
	
		We note that
		\begin{align}
			\lim_{\rho\to0}	\Entr^{(\rho)}(f(t)) = \Entr(f(t)) \text{ for all }t\ge0.
		\end{align}		
		Furthermore,  monotone convergence gives
		\begin{align}
			\lim_{\rho\to0}	
			\int_0^t\!\int_{\dom} \frac{1}{h(f)}|\nabla f+h(f)v|^2\vp_\rho(|v|)\,\dd v\dd\tau
			=	\int_0^t\!\int_{\dom} \frac{1}{h(f)}|\nabla f+h(f)v|^2\,\dd v\dd\tau.
		\end{align}
		Hence, it remains to prove that  the quantity
			\begin{align}
			B(\tau,\rho):=\int_{\dom}[\,\tfrac{1}{2}|v|^2+\Phi'(f)\,]\,(\nabla f+h(f)v)\cdot\tfrac{v}{|v|}\vp_\rho'(|v|)\,\dd v
		\end{align}
		satisfies
		\begin{align}\label{eq:105}
			\lim_{\rho\to0}	\int_0^t B(\tau,\rho)\,\dd\tau = 0.
		\end{align}
	Using the isotropy of $f(\tau,\cdot)$ we write 
		\begin{align}
		B(\tau,\rho)=c_d\int_0^{2\rho}[\,\tfrac{1}{2}r^2+\Phi'(g)\,]\,(\partial_r g+h(g)r)\vp_\rho'(r)\,r^{d-1}\dd r,
	\end{align}
	where $c_d$ denotes the area of the unit sphere. We can now argue similarly as in the proof of Theorem~\ref{thm:renorm} (see page~\pageref{page:Alim.ren}).
The function $|B(\tau,\rho)|$ is uniformly bounded  for $(\tau,\rho)\in[0,t]\times (0,1]$ thanks to the estimate
	$|(\partial_rg+rh(g))r^{d-1}|\le K_*$ and the fact that, 
	by~\ref{eq:hpinit} and Proposition~\ref{prop:lim}~\ref{it:dec}, $\inf_{[0,t]\times(0,1]}g=: \iota>0$.
	Indeed, note that for $(\tau,\rho)\in[0,t]\times (0,1]$
	\begin{align}
			|B(\tau,\rho)|\le C(\Phi'(\iota))K_*\int_0^{2\rho}|\vp'(\rho^{-1}r)|\rho^{-1}\dd r
			=C(\Phi'(\iota))K_*\int_0^{2}|\vp'(\hat r)|\dd \hat r.
	\end{align}
	Identity~\eqref{eq:105} therefore follows from the dominated convergence theorem provided we can prove the pointwise convergence 
	$\lim_{\rho\to0}B(\tau,\rho)=0$ for a.e.\ $\tau>0$.
	\smallskip
	
	\underline{Case 1: $g(\tau,0+)=+\infty$.}
	In this case, we estimate
	\begin{align}
		|B(\tau,\rho)|\le C\kcp \int_0^{2}|\vp'(\hat r)|\dd \hat r
		\cdot\sup_{r\in(0,\rho)}|\tfrac{1}{2}r^2+\Phi'(g(\tau,r))|
	\end{align}
and note that the sublinearity of $\Phi$ at infinity implies
	\begin{align}
		\sup_{r\in(0,\rho)}|\tfrac{1}{2}r^2+\Phi'(g(\tau,r))|\to0\quad \text{ as }\rho\to0.
	\end{align}
	Hence, $\lim_{\rho\to0}	|B(\tau,\rho)|=0$.
	\smallskip
	
	\underline{Case 2: $g(\tau,\cdot)\in L^\infty$.}
	Here, the assertion $\lim_{\rho\to0}|B(\tau,\rho)|=0$ is obtained similarly
	as in Case 2 of the proof of Theorem~\ref{thm:renorm} using the regularity of $g(\tau,\cdot)$ shown in Theorem~\ref{thm:profile.r}.
	\end{proof}
	
\section{Long-time behaviour}\label{sec:longtime}

\subsection{Relaxation to equilibrium}\label{ssec:relax}

\begin{proof}[Proof of Theorem~\ref{thm:long-time}]	
	Let
	\[
	\mathcal{D}(f):=\int_\dom \frac{1}{h(f)}|\nabla f+h(f)v|^2\,\dd v=\int_\dom h(f)|\nabla\delta \Entr(f)|^2\,\dd v.
	\]	
	Proposition~\ref{prop:edb} implies that
	$\int_0^\infty\mathcal{D}(f(t))\,\dd t\le \mathcal{H}(\fin)-\inf_{\mathcal{M}_+}\mathcal{H}<\infty$, and hence
	there exists an increasing sequence $t_k\to\infty$ such that 
	\begin{align}\label{eq:D->0}
		\lim_{k\to\infty}\mathcal{D}(f(t_k))=0.
	\end{align}

The sequence $\{\mu_{t_k}\}_k\subset\mathcal{M}_+(\mathbb{R}^d)$ of measures of mass $m$ is tight since $\sup_k\int|v|^3\dd\mu_{t_k}\le C\|\fin\|_{L^1_3}<\infty$. Prokhorov's theorem thus ensures the existence of a measure
$\mu_\infty\in\mathcal{M}_+(\mathbb{R}^d)$ with $\int_\dom\dd\mu_\infty=m$ such that, along a subsequence (not relabelled),
	 \begin{align}\label{eq:598}
	 	\mu_{t_k}&\weakstar\mu_\infty\quad\text{ in }\mathcal{M}_+(\mathbb{R}^d),
	 	\\\hspace{-3cm}\text{ and moreover, }\qquad\int|v|^2\dd\mu_{t_k}&\to\int|v|^2\dd\mu_\infty.
	 \end{align} 
 	At the same time, by Proposition~\ref{prop:lim}~\ref{it:dec}, $\mu_{t_k}=a(t_k)\delta_0+f(t_k)\mathcal{L}^d$
 	where $f$ satisfies a time-uniform bound of the form  $|f(t,r)|\le C(\rho)$ for all $r\ge\rho>0$ (cf.\ L.~\ref{l:boundiso} resp.\ Cor.~\ref{cor:boundaniso}).
 	Thus, the sequence $f_k(t):=f(t_k+t)$ obeys an estimate 
 	analogous to~\eqref{eq:597} for $G\subset\subset(-1,\infty)\times(\mathbb{R}^d\setminus\{0\})$, where we assume without loss of generality that $t_1\ge1$.
 	 We may therefore argue similarly as in the proof of Proposition~\ref{prop:lim} and invoke the Arzel\`a--Ascoli theorem to infer the existence of $f_\infty\in C^2(\mathbb{R}^d\setminus\{0\})\cap L^1_3(\mathbb{R}^d)$ such that, after possibly passing to another subsequence, 
 	\begin{align}\label{eq:792}
 		f(t_k)\to f_\infty\quad\text{ in }C^2_\loc(\mathbb{R}^d\setminus\{0\}).
 	\end{align}
Notice that $\mu_\infty^\mathrm{reg}=f_\infty\mathcal{L}^d$ and $\supp\mu_\infty^\mathrm{sing}\subseteq\{0\}$, as a consequence of~\eqref{eq:598} and~\eqref{eq:792}.
 
 We will now show that $\mu_\infty$ agrees with the minimiser of $\mathcal{H}$ of mass $m$.
To this end, let 
$\lambda(s)=\exp(\Phi'(s))=\frac{s}{(1+s^\gamma)^{1/\gamma}}$
and note that $\frac{h(s)}{\lambda^2(s)} 
= \frac{s(1+s^\gamma)^{1+2/\gamma}}{s^2}\ge (1+s^\gamma)^{1+\frac{1}{\gamma}}\ge 1$ for all $s\in(0,\infty)$. 
 Hence, 
 \begin{align}
 	\mathcal{D}(f)&=\int_{\mathbb{R}^d}h(f)|\nabla\Phi'(f)+v|^2\,\dd v
 	\\&=\int_{\mathbb{R}^d}\tfrac{h(f)}{\lambda(f)^2}|\lambda(f)\nabla\Phi'(f)+\lambda(f)v|^2\,\dd v
 	\\&\ge\int_{\mathbb{R}^d}|\nabla\lambda(f)+v\lambda(f)|^2\,\dd v.
 \end{align}
Thanks to the convergence~\eqref{eq:D->0}, the last estimate applied to $f:=f(t_k)$ implies that
\begin{align}
	\nabla\lambda(f(t_k))+v\lambda(f(t_k))\to 0\quad\text{ in }L^2(\mathbb{R}^d). 
\end{align}
At the same time, using~\eqref{eq:792}, we may pass to the limit in the sense of distributions
 \begin{align}
 	\nabla\lambda(f(t_k))+v\lambda(f(t_k))\to	\nabla\lambda(f_\infty)+v\lambda(f_\infty)\quad\text{ in }\mathscr{D}'(\mathbb{R}^d{\setminus}\{0\}). 
 \end{align}
 Thus, $\nabla\lambda(f_\infty)+v\lambda(f_\infty)=0$ in $\mathscr{D}'(\mathbb{R}^d{\setminus}\{0\})$, and hence 
 $\nabla\big(\ee^{\frac{1}{2}|v|^2}\lambda(f_\infty)\big)=0$ in $\mathbb{R}^d{\setminus}\{0\}$.
This implies that $\Phi'(f_\infty)+\frac{1}{2}|v|^2\equiv-\theta$ for a constant $\theta\in\mathbb{R}_{\ge0}$, where the sign of $\theta$ follows from the fact that $\Phi'\le0$.
Hence, $f_\infty(v)=(\Phi')^{-1}(-\frac{1}{2}|v|^2-\theta)=f_{\infty,\theta}(v)$.
To determine $\theta$,  recall that $\mu_\infty(\mathbb{R}^d)=m$. Thus, if $\|f_{\infty,\theta}\|_{L^1(\mathbb{R}^d)}<m$, then $\mu_\infty(\{0\})>0$.
 In this case,  the convergence~\eqref{eq:598} combined with the time-uniform upper bound $\sup_{t>0}f(t,v)\le C|v|^{-\frac{2}{\gamma}}\in L^1(B_{r_*})$, which follows from Theorem~\ref{thm:profile.r} (resp.\ Cor.~\ref{cor:upper.ani}), implies the existence of $\underline k\in \mathbb{N}$ such that  $\mu_{t_k}(\{0\})>0$ for all $k\ge\underline k$. Invoking once more Theorem~\ref{thm:profile.r} (now using the radial symmetry assumption), we find that~\eqref{eq:profile*} holds true for all such $t_k$ and, owing to the convergence~\eqref{eq:792}, we conclude that $\theta=0$. 
 If on the other hand $\|f_{\infty,\theta}\|_{L^1(\mathbb{R}^d)}=m$, there is no excess mass and we must have $\mu_\infty=f_{\infty,\theta}\mathcal{L}^d$.
In conclusion, we have shown that the measure $\mu_\infty$ coincides with the unique minimiser 
 $\mu_{\min}=\mu_{\min}^{(m)}$ of mass $m$.
 
 From the convergence properties established so far we infer $\lim_{k\to\infty}\Entr(\mu_{t_k})=\Entr(\mu_{\min})$. Since $t\mapsto\Entr(\mu_t)$ is non-increasing, this immediately yields
 \begin{align}
 	\lim_{t\to\infty}\Entr(\mu_t)=\Entr(\mu_{\min}).
 \end{align}
 Combining this result with the above compactness properties,  mass conservation, and the uniqueness of the minimiser $\mu_{\min}^{(m)}$, one can easily deduce the remaining 
 convergence properties along any sequence $t\to\infty$ as asserted in Theorem~\ref{thm:long-time}. Here, also recall the bounds $\sup_{t>0}f(t,v)\lesssim |v|^{-\frac{2}{\gamma}}$ for $|v|\le r_*$, $f(t,v)\lesssim 1$ for $|v|\ge r_*$ and
 $\sup_{t>0}\|f(t)\|_{L^1_3}\lesssim\|\fin\|_{L^1_3}$, which guarantee that $\lim_{t\to\infty}\mu_t(\{0\})=\mu_\infty(\{0\})$ and 
 $\lim_{t\to\infty}\|f(t)-f_{\infty,\theta}\|_{L^p(\mathbb{R}^d)}=0$ for $p\in[1,\frac{\gamma d}{2})$.
\end{proof}

\subsection{Long-time and transient properties}\label{ssec:dynprop}
Let us briefly point out some  implications of the above analysis on further qualitative dynamical properties, restricting for consistency to the isotropic case. 
If $m<m_c$, Theorem~\ref{thm:long-time} along with Theorem~\ref{thm:profile.r} imply the eventual regularity of $\mu_t$ after some sufficiently large time $T\gg1$. However, using a contradiction argument, finite-time blow-up and the formation of a condensate (that is $\mu_t(\{0\})>0$ for some $t>0$)
can be shown to occur for any size of the mass $m>0$ by choosing the smooth initial data sufficiently concentrated near the origin (cf.~\cite{CHR_2020,Toscani_2012}). 
Hence,  there exist flows exhibiting \textit{transient condensates} with singular parts compactly supported in time.
On the other hand, whenever $m>m_c$, the above theory implies the eventual formation of a condensate: $\exists T\gg1$ such that $\mu_{t}(\{0\})>0$ for all $t\ge T$. This is a consequence of the convergence $\lim_{t\to\infty}\mu_{t}(\{0\})=\mu_{\min}(\{0\})$.
It is also possible to infer information on the spatiotemporal features of singularity formation 
and regularisation using rescaling methods. We refer to~\cite[Chapter~5.2]{hopf_PhDThesis_2019}, where such dynamics have been shown to be of \enquote{type II} for the 1D case.

Finally, we note that finite-time condensation in the mass-supercritical case and convergence to the entropy minimiser can also be deduced in the anisotropic setting if $\fin$ admits a mass-supercritical isotropic lower barrier,
 i.e.\ $\fin\ge \fin^\#$ for some non-negative radially symmetric function $\fin^\#$ with $\int\fin^\#> m_c$.
 In this case, the density $f(t,\cdot)$ is squeezed between two isotropic barriers which, in virtue of Theorem~\ref{thm:long-time}, both converge to $f_c$ as $t\to\infty$.

\subsection{Concluding remark}

The comparison principle structure provides us with a priori bounds that allow for a detailed characterisation of the singularities which isotropic flows starting from regular data may exhibit (and even gives uniqueness in the 1D case~\cite{CHR_2020,hopf_PhDThesis_2019} resp.\ convergence of the scheme to a unique limit in higher dimensions). 
However, one may not expect such a structure to persist in more complex situations.
Particularly with regard to the study of uniqueness and stability properties in the presence of singularities, it would be interesting to see whether variational problems like~\eqref{eq:fpn} 
allow for more robust approaches.

\appendix

\section{}
\label{app:FPsemigr}

Recall that $	\mathcal{F}(t,v,w) = \ee^{dt}G_{\nu(t)}(\ee^tv-w)$, $\nu(t)=\mathrm{e}^{2t}-1$, 
	$G_\lambda(\xi)=(2\pi\lambda)^{-\frac{d}{2}}\mathrm{e}^{-\frac{|\xi|^2}{2\lambda}}$.

\begin{lemma}\label{l:rdfp}
	 Let $T\le 1$ and let $\tilde q\in[1,\infty]$. There exists $C=C(\tilde q, d)<\infty$ such that for all $t\in(0,T]$
 \begin{align}\label{eq:rdfp}
	\begin{aligned}
		\bigg\|\int_{\mathbb{R}^d}|\nabla_v\mathcal{F}(t,v,w)||\ee^{-t}w-v| |f(w)|\,\dd w\bigg\|_{L^{\tilde q}(\mathbb{R}^d)}
		&\le C\|f\|_{L^{\tilde q}(\mathbb{R}^d)}
		\\&\le C\nu(t)^{-\frac{1}{2}}\|f\|_{L^{\tilde q}(\mathbb{R}^d)}.
	\end{aligned}
\end{align}
\end{lemma}
\begin{proof}The second bound in~\eqref{eq:rdfp} is trivial. 
	
	To verify the first inequality, we compute for $t\in(0,T]$
\begin{align}
	\int_{\mathbb{R}^d}|\nabla_v&\mathcal{F}(t,v,w)||\ee^{-t}w-v| |f(w)|\,\dd w
	\\&=(2\pi\nu(t))^{-\frac{d}{2}}\ee^{dt}\ee^{t}(2\nu(t))^{-\frac{1}{2}}\int_{\mathbb{R}^d}2\tfrac{|\ee^{t}v-w|}{\sqrt{2\nu(t)}}
	\mathrm{e}^{-\frac{|\ee^tv-w|^2}{2\nu(t)}}
	|\ee^{-t}w-v| |f(w)|\,\dd w
	\\&=(2\pi\nu(t))^{-\frac{d}{2}}\ee^{2dt}\ee^{t}(2\nu(t))^{-\frac{1}{2}}
	\int_{\mathbb{R}^d}2\frac{|\ee^{t}(v-\tilde w)|}{\sqrt{2\nu(t)}}
	\mathrm{e}^{-\frac{|\ee^t(v-\tilde w)|^2}{2\nu(t)}}
	|\tilde w-v| |f(\ee^t\tilde w)|\,\dd \tilde w 
	\\&=(2\pi\nu(t))^{-\frac{d}{2}}\ee^{2dt}
	\int_{\mathbb{R}^d}2\frac{|\ee^{t}(v-\tilde w)|^2}{2\nu(t)}
	\mathrm{e}^{-\frac{|\ee^t(v-\tilde w)|^2}{2\nu(t)}}
	|f(\ee^t\tilde w)|\,\dd \tilde w.
\end{align}
Now, the asserted inequality follows upon an application of Young's convolution inequality, $\|a\ast b\|_{L^{\tilde q}}\le \|a\|_{L^1}\|b\|_{L^{\tilde q}}$.
\end{proof}

\section*{Acknowledgements}
The author would like to thank Prof.\ John King for valuable discussions on this topic.
	
	\newcommand{\etalchar}[1]{$^{#1}$}

\end{document}